\documentclass[sn-mathphys]{sn-jnl}

\usepackage{amsmath,amscd}
\usepackage{fancyhdr}

\numberwithin{equation}{section}

\jyear{2021}%

\theoremstyle{thmstyleone}%
\newtheorem{theorem}{Theorem}[section]
\newtheorem{proposition}[theorem]{Proposition}%
\newtheorem{lemma}[theorem]{Lemma}%
\newtheorem{remark}[theorem]{Remark}
\newtheorem{definition}[theorem]{Definition}

\begin{document}

\title[The Monotonicity Conjecture and Stability of Solitons ]{The Monotonicity Conjecture and Stability of Solitons for the Cubic-Quintic NLS on $\mathbb{R}^3$}


\author*[1]{\fnm{Jian} \sur{Zhang}}\email{zhangjian@uestc.edu.cn}

\author[2]{\fnm{Shihui} \sur{Zhu}}\email{shihuizhumath@163.com}

\affil*[1]{\orgdiv{School of Mathematical Sciences}, \orgname{University of Electronic Science and Technology of China}, \orgaddress{\street{} \city{Chengdu}, \postcode{611731}, \state{} \country{China}}}

\affil[2]{\orgdiv{VCVR Key Lab of Sichuan Province}, \orgname{Sichuan Normal University}, \orgaddress{\street{} \city{Chengdu}, \postcode{610066}, \state{} \country{China}}}


\abstract{In this paper, we prove stability or instability of solitons for the cubic-quintic nonlinear Schr\"{o}dinger equation (NLS) at every frequency. The monotonicity conjecture raised by Killip, Oh, Pocovnicu and Visan is resolved. We introduce and solve a new cross-constrained variational problem. Then uniqueness of the energy minimizers is proposed and shown. According to a spectral approach and  variational arguments, we develop a set of geometric analysis methods. Correspondences between the soliton frequency and the prescribed mass are established. Classification of normalized solutions is given.
}



\pacs[MSC classfication(2010)]{35Q55, 35B35, 37K40, 35J50, 35J60}

\maketitle
\newpage
\tableofcontents 

\pagestyle{fancy} 
\fancyhf{}
\fancyhead[OL]{ }
\fancyhead[OC]{}
\fancyhead[OR]{J. Zhang, S.  Zhu}
\fancyhead[EL]{The Monotonicity Conjecture and Stability of Solitons}
\fancyhead[EC]{ }
\fancyhead[ER]{ }
\fancyfoot[RO, LE]{\thepage}

\section{Introduction}\label{sec1}

This paper studies the cubic-quintic nonlinear Schr\"{o}dinger equation (NLS)
\begin{equation}\label{1.1}
i\partial_{t}\varphi +\Delta  \varphi +\lvert\varphi\rvert^2\varphi-\lvert\varphi\rvert^4\varphi=0,\quad (t,x)\in\mathbb{R}\times\mathbb{R}^3.
\end{equation}
(\ref{1.1}) is the defocusing energy critical  nonlinear Schr\"{o}dinger equation with focusing cubic perturbation on $\mathbb{R}^3$. We are interested in orbital stability of solitons at every frequency as well as the monotonicity conjecture raised by Killip, Oh, Pocovnicu and Visan \cite{KOPV2017}.

(\ref{1.1})  is a typical model for soliton theory (see \cite{KOPV2017,M2019}) in physics, such as nonlinear optics (see \cite{MMCTBMT2002,MMCML2000,DMM2000}), Langmuir waves in plasma (see \cite{ZH1994}) and the Bose-Einstein condensation (see \cite{GFTLC2000,AGTF2001}).

From viewpoint of mathematics, the focusing cubic or quintic nonlinear Schr\"{o}dinger equation on $\mathbb{R}^3$ possesses  solitons, but all solitons are unstable (see \cite{BC1981,S2009,KM2006,NS2012,OT1991}). The defocusing cubic or quintic nonlinear Schr\"{o}dinger equation on $\mathbb{R}^3$ has no any solitons to exist, but possesses scattering property (see \cite{B1999,CKSTT2008,D2012,DHR2008}). In addition, the scaling invariance of the pure power case is breaked in (\ref{1.1})(see \cite{C2003}). Therefore the dynamics of (\ref{1.1}) becomes a challenging issue (see \cite{KOPV2017,CS2021,MXZ2013,LR2020,TVZ2007,JL2022,AM2022,MKV2021}). This motivates us to study the stability of solitons for (\ref{1.1}) (see \cite{T2009}), which  directly  concerns about the monotonicity conjecture proposed by Killip, Oh, Pocovnicu and Visan (see Conjecture 2.3 in \cite{KOPV2017}).

In the energy space $H^{1}(\mathbb{R}^3)$, consider the scalar field equation for $\omega\in\mathbb{R}$,
\begin{equation}\label{1.2}
-\Delta u-\lvert u\rvert^2u+\lvert u\rvert^4u+\omega u =0, \quad u\in H^{1}(\mathbb{R}^3).
\end{equation}
From Berestycki and Lions \cite{BL1983}, if and only if
\begin{equation}\label{1.3}
0<\omega<\frac{3}{16},
\end{equation}
(\ref{1.2}) possesses non-trivial solutions (also see \cite{KOPV2017}). From Gidas, Ni and Nirenberg \cite{GNN1979}, every positive solution of (\ref{1.2}) is radially symmetric. From Serrin and Tang \cite{ST2000}, the positive solution of (\ref{1.2}) is unique up to translations. Therefore one concludes that for $\omega\in(0,\frac{3}{16})$, (\ref{1.2}) possesses a unique positive solution $Q_\omega(x)$ (see \cite{KOPV2017,K2011}), which is called ground state of (\ref{1.2}).

Let $Q_\omega$ be a ground state of (\ref{1.2}) with $\omega\in(0,\frac{3}{16})$. It is easily checked that
\begin{equation}\label{1.9}
\varphi(t,x)=Q_\omega(x)e^{i\omega t}
\end{equation}
is a solution of (\ref{1.1}), which is called a ground state soliton of (\ref{1.1}). We also directly call (\ref{1.9}) soliton of (\ref{1.1}), and call $\omega$ frequency of soliton. It is known that (\ref{1.1}) admits time-space translation invariance, phase invariance and Galilean invariance. Then for arbitrary $x_0\in\mathbb{R}^3$, $v_0\in\mathbb{R}^3$ and $\nu_0\in\mathbb{R}$, in terms of (\ref{1.9}) one has that
\begin{equation}\label{1.10}
 \varphi(t,x)=Q_\omega(x-x_0-v_0t)e^{i(\omega t+\nu_0+\frac{1}{2}v_0x-\frac{1}{4}\lvert v_0\rvert^2t)}
\end{equation}
is also a soliton of (\ref{1.1}).
By \cite{KOPV2017,S2021,CS2021,JJTV2022,JL2022,LR2020}, orbital stability of solitons with regard to every frequency  for (\ref{1.1}) is a crucial open problem.

So far there are two ways to study stability of solitons for nonlinear  Schr\"{o}dinger equations (refer to \cite{L2009}). One is variational approach originated from Cazenave and Lions \cite{CL1982}. The other is spectrum approach originated from Weinstein \cite{W1985,W1986} and then considerably generalized by Grillakis, Shatah and Strauss \cite{GSS1987,GSS1990}. Both approaches have encountered essential difficulties to (\ref{1.1}), since (\ref{1.1}) fails in both scaling invariance and effective spectral analysis. We need develop new methods to study stability of solitons for (\ref{1.1}).

In the following, we denote $\int_{\mathbb{R}^3}\cdot\ dx$ by $\int \cdot \ dx$ and $\lvert\lvert \ \cdot \ \rvert \rvert_{L^2(\mathbb{R}^3)}$ by $\lvert\lvert \ \cdot\ \rvert \rvert_{L^2}$ .
For $u\in H^1(\mathbb{R}^3)\backslash\{0\}$, define the Weinstein functional (refer to \cite{W1983})
\begin{equation}\label{1.4}
F(u)=\frac{\lvert\lvert u\rvert \rvert_{L^2}\lvert\lvert u\rvert \rvert_{L^6(\mathbb{R}^3)}^{\frac{3}{2}}\lvert\lvert \nabla u\rvert \rvert_{L^2}^{\frac{3}{2}}}
{\lvert\lvert u\rvert \rvert_{L^4(\mathbb{R}^3)}^{4}}
\end{equation}
and the variational problem
\begin{equation}\label{1.5}
d_0=\mathop{\mathrm{inf}}_{\{u\in H^1(\mathbb{R}^3)\backslash\{0\}\}} F(u).
\end{equation}
In \cite{KOPV2017}, Killip, Oh, Pocovnicu and Visan show that (\ref{1.5}) is achieved at some positive symmetric minimizer $\phi$. Thus $\phi$ satisfies the  Euler-Lagrange equation corresponding to (\ref{1.5}), which is the same as (\ref{1.2}) with certain $\omega=\omega_\phi\in(0,\frac{3}{16})$. It is obvious that uniqueness of the positive minimizers of (\ref{1.5}) is uncertain. But all positive minimizers of (\ref{1.5}) have the common mass (see \cite{KOPV2017})
\begin{equation}\label{1.6}
\rho=\lvert\lvert \phi\rvert \rvert_{L^2}^{2}=\frac{64}{9}d^2_0.
\end{equation}

For $u\in H^1(\mathbb{R}^3)$, define the mass functional
\begin{equation}\label{1.29}
M(u)=\int\lvert u\lvert^{2}dx,
\end{equation}
and define the energy functional
\begin{equation}\label{1.7}
E(u)=\int \frac{1}{2}\lvert \nabla u\rvert^2-\frac{1}{4}\lvert u\rvert^4+\frac{1}{6}\lvert u\rvert^6dx.
\end{equation}
Let $m>0$, and consider the variational problem
\begin{equation}\label{1.8}
d_m=\mathop{\mathrm{inf}}_{\{u\in H^1(\mathbb{R}^3),\;\;  M(u)=m\}} E(u).
\end{equation}
For the common mass $\rho$ in (\ref{1.6}), Killip, Oh, Pocovnicu and Visan \cite{KOPV2017} show the following results. When
\begin{equation}
0<m<\rho,
\end{equation}
(\ref{1.8}) is not achieved. When
\begin{equation}
m\geq\rho,
\end{equation}
(\ref{1.8}) is achieved at some positive minimizer $\psi_m$. Thus for $m\geq\rho$,
there exists the Lagrange multiplier $\omega_{\psi_m}\in\mathbb{R}$ such that $\psi_m$ satisfies the corresponding Euler-Lagrange equation, which is the same as (\ref{1.2}) with $\omega=\omega_{\psi_m}$.  It is clear that uniqueness of the positive minimizers of (\ref{1.8}) with $m\geq \rho$ is also uncertain. 

For $u\in H^{1}(\mathbb{R}^{3})$, define the Pohozaev functional (refer to \cite{P1965})
\begin{equation}\label{1.30}
I(u)=\int \frac{1}{3}\lvert \nabla u\lvert^{2}-\frac{1}{4}\lvert u\lvert^{4}
+\frac{1}{3}\lvert u\lvert^{6}dx.
\end{equation}
Let $m>0$, define the variational problem
\begin{equation}\label{1.31}
d_{m}^{I}=\inf_{\{u\in H^{1}(\mathbb{R}^{3}), \; M(u)=m, \; I(u)=0\}}E(u).
\end{equation}
For the common mass $\rho$ in (\ref{1.6}), Killip, Oh, Pocovnicu and Visan \cite{KOPV2017} get the following results. When
\begin{equation}\label{1.32}
0<m<\frac{4}{3\sqrt{3}} \rho ,
\end{equation}
one has that $d_{m}^{I}=\infty$ and (\ref{1.31}) is not achieved. When
\begin{equation}\label{1.33}
\frac{4}{3\sqrt{3}} \rho\leq m<\rho ,
\end{equation}
one has that $0<d_{m}^{I}<\infty$ and (\ref{1.31}) is achieved at some positive minimizer $\psi$. When
\begin{equation}
m\geq \rho,
\end{equation}
one has that $d_{m}^{I}=d_{m}$ and (\ref{1.31}) is achieved at some positive minimizer $\psi_{m}$. Thus for $m\geq \frac{4}{3\sqrt{3}} \rho$,  (\ref{1.31}) is always achieved at some positive minimizer $\psi$,  we call it the energy minimizer for (\ref{1.1}). It follows that there exist the Lagrange multipliers $\mu\geq 0$ and $\nu\in\mathbb{R}$ such that $\psi$ satisfies the Euler-Lagrange equation
\begin{equation}\label{1.35}
-\Delta \psi+\psi^{5}-\psi^{3}+\mu(-2\Delta \psi+6\psi^{5}-3\psi^{3})+2\nu \psi=0.
\end{equation}
Under the change of variables $\psi(x)=av(\lambda x)$ with
\begin{equation}\label{1.36}
a^{2}=\frac{1+3\mu}{1+6\mu},\;\; \lambda^{2}=\frac{(1+3\mu)^{2}}{(1+2\mu)(1+6\mu)},\;\;\; \omega=\frac{2\nu(1+6\mu)}{(1+3\mu)^{2}},
\end{equation}
(\ref{1.35}) becomes
\begin{equation}\label{1.37}
-\Delta v+v^{5}-v^{3}+\omega v=0.
\end{equation}
Then one concludes that $v=Q_{\omega}$ up to a translation for certain $\omega\in (0, \frac{3}{16})$. Moreover we conclude that there exists
\begin{equation}\label{1.120}
m_{0}=\inf_{\{0<\omega<\frac{3}{16}\}}M(Q_{\omega})
\end{equation}
such that
\begin{equation}\label{1.38}
\frac{4}{3\sqrt{3}} \rho\leq m_{0}\leq\rho.
\end{equation}

It is clear that uniqueness of the positive minimizers of (\ref{1.31}) with $m\geq \frac{4}{3\sqrt{3}} \rho$ is still uncertain. But for certain  minimizer $\psi$ of $d_{m}^{I}$, one can uniquely determine the Lagrange multiplier $\nu$ by (\ref{1.35}). Then in terms of $\psi$ and $\nu$, we introduce the variational problem
\begin{equation}\label{1.11}
d:=\mathop{\inf}\limits_ {\{u\in H^1(\mathbb{R}^{3}),\  0<\int\lvert u\rvert^2dx\leq\int
\lvert \psi\rvert^2dx,\  I(u)=0\}
}
[E(u)+\nu\int\lvert u\rvert^2dx].
\end{equation}
We see that $d$ is related to $d_{m}^{I}$ and the functional $E(u)+\nu\int\lvert u\rvert^2dx$ is constrained to the flow $\{u\in H^1(\mathbb{R}^{3}),\  0<\int\lvert u\rvert^2dx\leq\int
\lvert \psi\rvert^2dx\}$ with $I(u)=0$. Therefore we call (\ref{1.11}) the cross-constrained variational problem.

We will prove that the infimum $d$ in (\ref{1.11}) is achieved and $\psi$
is the unique positive solution of (\ref{1.11}). By this result, we can prove that the energy minimizer for (\ref{1.1}), that is the positive solution of $d_{m}^{I}$, is unique. Furthermore we can prove that both the positive minimizer of (\ref{1.5}) and the positive minimizer of (\ref{1.8}) are also unique. Thus we get that the variational problems (\ref{1.5}), (\ref{1.8}), (\ref{1.31}) and (\ref{1.11}) all possess uniqueness of the positive minimizers.

For $m\geq \frac{4}{3\sqrt{3}}\rho$, let $\psi$ be the unique energy minimizer, that is the positive minimizer of $d_{m}^{I}$. Then $\psi$ satisfies the Euler-Lagrange equation (\ref{1.35}). We conduct the change of variables $\psi(x)=av(\lambda x)$ with (\ref{1.36}). Then $v$ satisfies (\ref{1.37}). Thus we concludes that $v=Q_{\omega}$ up to a translation for some $0<\omega<\frac{3}{16}$, which we call the energy ground state. For every $m\geq \frac{4}{3\sqrt{3}}\rho$, we get a unique energy minimizer $\psi$ and a unique ground state $v=Q_{\omega}$ up to a translation for some $0<\omega<\frac{3}{16}$.
Then for $m_{0}$ defined by (\ref{1.120}), according to \cite{KOPV2017} and the above uniqueness of the positive minimizers, we
can uniquely determine $\omega_{\ast}\in (0, \frac{3}{16})$ such that
\begin{equation}\label{1.121}
M(Q_{\omega_{\ast}})= m_{0}=\inf_{\{0<\omega<\frac{3}{16}\}}M(Q_{\omega}).
\end{equation}
In terms of Cazenave and Lions' variational arguments \cite{CL1982}, we can prove that the energy ground states are orbitally stable. Of course we now can not determine for which $\omega$ the energy ground state is orbitally stable in general. But specially we can get that $Q_{\omega_{\ast}}$ is orbitally stable.

Now we develop a method, which is parallel to the complex analysis, to prove that the zero point of $\frac{d}{d\omega}M(Q_{\omega})$ is isolated. Based on this result, according to the spectral approach,  we develop a set of geometric analysis methods. Then we can prove the monotonicity conjecture and stability or instability of solitons at every frequency.

Our results are concerned with the normalized solution of (\ref{1.2}), which is defined as a non-trivial solution of (\ref{1.2}) satisfying the prescribed mass $\int\lvert u\lvert^{2}dx=m$ (see \cite{JL2021}). Besides motivations in mathematical physics, normalized solutions are also of interest in the framework of ergodic Mean Field Games system \cite{CV2017}. Recent studies on normalized solutions refer to \cite{BMRV2021,JL2021,WW2022,S2021} and the references there. We can establish correspondences between the soliton frequency and the prescribed mass. Then we get a classification of normalized solutions.


The main results of this paper read as follows:
\begin{theorem}\label{t1.1}
For the common mass $\rho$ introduced in (\ref{1.6}), let $m\geq\frac{4}{3\sqrt{3}}\rho$. Then the variational problems (\ref{1.5}), (\ref{1.8}), (\ref{1.31}) and (\ref{1.11}) all possess the unique positive minimizer up to translations. Moreover let $m_{0}=\inf_{\{0<\omega<\frac{3}{16}\}}M(Q_{\omega})$. Then $\frac{4}{3\sqrt{3}}\rho\leq m_{0}\leq\rho$ and there exists a unique $\omega_{\ast}\in (0, \frac{3}{16})$ such that $M(Q_{\omega_{\ast}})=m_{0}$.
\end{theorem}

\begin{theorem}\label{t1.3}
Let $\omega\in (0, \frac{3} {16})$ and $Q_{\omega}(x)$ be the positive solution of (\ref{1.2}). Then there exists a unique $\omega_{\ast}$ such that $0<\omega_{\ast}<\frac{3}{16}$. When $\omega\in (0, \omega_{\ast})$, the soliton $Q_{\omega}(x)e^{i\omega t}$ of (\ref{1.1}) is orbitally unstable; When $\omega\in [\omega_{\ast}, \frac{3}{16})$, the soliton $Q_{\omega}(x)e^{i\omega t}$ of (\ref{1.1}) is orbitally stable.
\end{theorem}

\begin{theorem}\label{t1.2}
Let $M(Q_{\omega})=\int Q_{\omega}^{2}dx$ for the positive solution $Q_{\omega}$ of (\ref{1.2}) with $\omega\in (0, \frac{3} {16})$. Then there is an $\omega_*\in (0, \frac{3} {16})$ so that $\omega\rightarrow M(Q_{\omega})$ is strictly decreasing for $\omega<\omega_{\ast}$ and strictly increasing for $\omega>\omega_{\ast}$.
\end{theorem}

\begin{theorem}\label{t1.4}
There exists $m_{0}$ such that $\frac{4}{3\sqrt{3}}\rho\leq m_{0}\leq\rho$. When $m=m_{0}$, (\ref{1.2}) has a unique positive normalized solution with the prescribed mass $\int\lvert u\lvert^{2}dx=m_{0}$, which is just $Q_{\omega_{\ast}}$, where $\omega_{\ast}$ is the same as Theorem \ref{t1.3} and Theorem \ref{t1.2}. When $m> m_{0}$, (\ref{1.2}) just has two positive normalized solutions with the prescribed mass $\int\lvert u\lvert^{2}dx=m$. The one is $Q_{\omega_{1}}$ with $\omega_{1}\in (0, \omega_{\ast})$ and the other is $Q_{\omega_{2}}$ with $\omega_{2}\in (\omega_{\ast}, \frac{3}{16})$. When $m<m_{0}$, (\ref{1.2}) has no positive normalized solution with the prescribed mass $\int\lvert u\lvert^{2}dx=m$.
\end{theorem}

Theorem \ref{t1.1} gets uniqueness of the positive minimizers for $d_{0}$, $d_{m}$, $d_{m}^{I}$ and $d$. Thus Theorem \ref{t1.1} settles the questions raised by \cite{JL2022} and \cite{KOPV2017}, where uniqueness of the energy minimizers is proposed. Here we still denote the positive minimizers of $d_{m}^{I}$ by the energy minimizers, and denote the ground states corresponding to the energy minimizers by the energy ground states, which satisfy (\ref{1.37}) for certain $\omega \in (0,\frac{3}{16})$. We see that uniqueness of the energy minimizers is concerned with the monotonicity of $M(Q_{\omega})$ for $\omega \in (0,\frac{3}{16})$ and stability of solitons. Especially by the uniqueness of the energy minimizers, we can determine uniquely critical frequency $\omega_{\ast} \in (0,\frac{3}{16})$ such that for all $\omega \in (0,\frac{3}{16})$ one has $M(Q_{\omega_{\ast}})\leq M(Q_{\omega}) $.

Theorem \ref{t1.3} establishes sharp results for orbital stability of solitons of (\ref{1.1}) with regard to every frequency by the critical frequency $\omega_{\ast}$. Therefore Theorem \ref{t1.3} settles the questions raised by \cite{KOPV2017,S2021,CS2021,JJTV2022,JL2022,LR2020}, where the orbitally stability of solitons is only with regard to the set of ground states. Then Theorem \ref{t1.3} also answers the open problem on stability of solitons proposed by Tao \cite{T2009}.

Theorem \ref{t1.2} completely verifies the monotonicity conjecture proposed by
Killip, Oh, Pocovnicu and Visan (see Conjecture 2.3 in \cite{KOPV2017}). In \cite{LR2020}, Lewin and Rota Nodari get the exact behavior of the mass $M(Q_{\omega})$ and its derivative $\frac{d}{d\omega}M(Q_{\omega})$, and show the monotonicity close to both $0$ and $\frac{3}{16}$. They give a proof of the monotonicity conjecture at large mass. In fact, Lewin and Rota Nodari \cite{LR2020} consider more general situations. Therefore Lewin and Rota Nodari improve a lot about the monotonicity conjecture and make some advances. The related studies we also refer to Carles and Sparber \cite{CS2021} as well as Jeanjean et al \cite{JJTV2022,JL2021,JL2022}.

We see that Theorem \ref{t1.2} strongly suggest that Theorem \ref{t1.3} be true, but Theorem \ref{t1.2} can not derive Theorem \ref{t1.3}. In turn, Theorem \ref{t1.3} can not also derive Theorem \ref{t1.2}. In this paper by introducing and solving a new cross-constrained variational problem, we prove uniqueness of the energy minimizers. Then by Theorem \ref{t1.1}, we can determine a unique critical frequency $\omega_{\ast} \in (0,\frac{3}{16})$ such that $M(Q_{\omega_{\ast}})\leq  M(Q_{\omega})$ for all $\omega \in (0,\frac{3}{16})$. Thus we  develop a set of geometric analysis methods to prove sharp stability of  solitons and the monotonicity conjecture. We need to point out that uniqueness of the energy minimizers plays a key role in the proofs of Theorem \ref{t1.3} and Theorem \ref{t1.2}. We make some new advances in variational methods and bridge between variational method and spectral analysis.

Theorem \ref{t1.4} gives a complete classification about normalized solutions of (\ref{1.2}), which depends on establishing two correspondences between the frequency and the mass. Thus Theorem \ref{t1.4} settles the questions raised by \cite{BMRV2021,JL2021, WW2022, S2021}.

We see that (\ref{1.1}) is almost a perfect nonlinear dispersive model, which is global well-posedness in energy space and possesses scattering, stable solitons as well as unstable solitons without blow up. By \cite{KOPV2017} and Theorem \ref{t1.3} in this paper, the global dynamics of (\ref{1.1}) is more comprehensively described (also see \cite{MKV2021}). According to \cite{CMM2011,MM2006,MMT2006}, by Theorem \ref{t1.3} in this paper, multi-solitons of (\ref{1.1})
can be constructed (also refer to \cite{CL2011,BZ2022}). But the soliton resolution conjecture for (\ref{1.1}) is still open (see \cite{T2009}).

This paper is organized as follows. In section 2, we state some propositions about classic variational problems as well as (\ref{1.1}) and (\ref{1.2}). In section 3, we introduce a new cross-constrained variational problem and solve it by the profile decomposition theory \cite{HK2005} and Schwartz symmetric rearrangement theory \cite{LL2000}. This is the basic idea of our work. In section 4, we prove uniqueness of the  positive minimizers for (\ref{1.31}) as well as uniqueness of the positive minimizers for both (\ref{1.5}) and (\ref{1.8}), which is the key result of our work. Then we get Theorem \ref{t1.1}.
In section 5, we prove stability of the energy ground states.
In section 6, we state some spectral properties about (\ref{1.1}) and (\ref{1.2}). In section 7 and 8, we combine the spectral approach with the variational method and develop a geometric analysis for the mass functional. Then we get sharp stability threshold of solitons  for (\ref{1.1}) and completely resolve the monotonicity conjecture proposed by Killip, Oh, Pocovnicu and Visan \cite{KOPV2017}. Thus we get the proofs of Theorem \ref{t1.3} and Theorem \ref{t1.2}. In section 9, by establishing correspondences between the frequency and the mass, we prove Theorem \ref{t1.4}.

\section{Classic variational problems}\label{sec2}


According to  Killip, Oh, Pocovnicu and Visan \cite{KOPV2017}, the following propositions are true.
\begin{proposition}\label{p2.3}
The variational problem (\ref{1.5}) is achieved at some positive minimizer $\phi$. In addition, every positive minimizer $\phi$ of (\ref{1.5}) satisfies (\ref{1.2}) with some $\omega=\omega_\phi\in(0,\frac{3}{16})$. Moreover all positive minimizers of (\ref{1.5}) have the common mass:
$$\rho=\lvert\lvert \phi\rvert \rvert_{L^2}^2=\frac{64}{9}d_0^2$$
and zero energy: $E(\phi)=0$.
For arbitrary $u\in H^1(\mathbb{R}^3)$, one has the Gagliardo-Nirenberg-H\"{o}lder inequality
$$ \lvert\lvert u\rvert \rvert^4_{L^4(\mathbb{R}^3)}\leq\frac{8}{3}\rho^{-\frac{1}{2}}
\lvert\lvert u\rvert \rvert_{L^2}\cdot \lvert\lvert u\rvert \rvert^{\frac{3}{2}}_{L^6(\mathbb{R}^3)}\cdot \lvert\lvert \nabla u\rvert \rvert^{\frac{3}{2}}_{L^2},
$$
and the interpolation inequality
$$ \lvert\lvert u\rvert \rvert^4_{L^4(\mathbb{R}^3)}\leq\frac{8}{3}\rho^{-\frac{1}{2}}
\lvert\lvert u\rvert \rvert_{L^2}\cdot \left[\frac{1}{4}\lvert\lvert u\rvert \rvert^{6}_{L^6(\mathbb{R}^3)}+\frac{3}{4} \lvert\lvert \nabla u\rvert \rvert^{2}_{L^2}\right].
$$
\end{proposition}
\begin{proposition}\label{p2.5}
For the variational problem (\ref{1.8}), one has that:\\
(I) If $0<m<\rho$, then
$$E(u)>0 \ and \ d_m=0.$$
In this case, $d_m$ is not achieved.\\
(II) If $m=\rho$, then
$$E(u)\geq0 \ and \ d_m=0.$$
In this case, $d_m$ is achieved at some positive minimizer $\psi_\rho$. In addition, every positive minimizer $\psi_\rho$ of (\ref{1.8}) satisfies (\ref{1.2}) with some $\omega=\omega_{\psi_\rho}\in(0,\frac{3}{16})$.\\
(III) If $m>\rho$, then
$$E(u)\geq d_m \ and \ d_m<0.$$
In this case, $d_m$ is achieved at some positive minimizer $\psi_m$. In addition, every positive minimizer $\psi_m$ of (\ref{1.8}) satisfies (\ref{1.2}) with some $\omega=\omega_{\psi_m}\in(0,\frac{3}{16})$.
\end{proposition}
\begin{proposition}\label{p2.2}
(\ref{1.2}) possesses a positive solution $Q_\omega$ if and only if $\omega\in(0,\frac{3}{16})$. In addition, $Q_\omega$  holds the following properties.\\
(I) $Q_\omega$ is radially symmetric and unique up to translations.\\
(II) $Q_\omega(x)$ is a real-analytic funtion of $x$ and that for some $c=c(\omega)>0$ as $\lvert x\rvert\rightarrow \infty$,
\begin{equation*}
\lvert x\rvert \mathrm{exp}\{\sqrt{\omega}\lvert x\rvert\}Q_\omega(x)\rightarrow c,\quad
\mathrm{exp}\{\sqrt{\omega}\lvert x\rvert\}x \nabla Q_\omega(x)\rightarrow -\sqrt{\omega}c.
\end{equation*}
(III) Pohozaev identity
\begin{equation*}
\int \frac{1}{3}\lvert \nabla Q_\omega \rvert^2 + \frac{1}{3}Q_\omega^6-\frac{1}{4}Q_\omega^4dx=0.
\end{equation*}
\end{proposition}
\begin{proposition}\label{p2.2a}
Let $\omega\in(0,\frac{3}{16})$ and $Q_\omega$ be the ground state of (\ref{1.2}). One has\\
(I) put $M(Q_\omega)=\int Q_\omega^2dx$, then
\begin{equation*}
M(Q_\omega)\rightarrow\infty \ as \ \omega\rightarrow0,\\
\end{equation*}
\begin{equation*}
M(Q_\omega)\rightarrow\infty \ and \ E(Q_\omega)\rightarrow-\infty, \ as \ \omega\rightarrow\frac{3}{16};
\end{equation*}
(II) the map $\omega\rightarrow M(Q_\omega)$ is $C^1$, indeed, $M(Q_\omega)$ is real analytic;\\
(III) put $\beta(\omega)=\int Q_\omega ^6dx$ / $\int \lvert \nabla Q_\omega \rvert^2dx$, then
\begin{equation*}
\begin{split}
&\frac{d}{d\omega}E(Q_\omega)=-\frac{\omega}{2}\frac{d}{d\omega}M(Q_\omega),\quad
\frac{d}{d\omega}\int \lvert \nabla Q_\omega \rvert^2dx=\frac{3}{2}M(Q_\omega),\\
&\int Q_\omega ^4dx=4\omega\int Q_\omega ^2dx, \qquad E(Q_\omega)=\frac{1-\beta(\omega)}{6}\int\lvert \nabla Q_\omega\rvert ^2dx,\\
&\int Q_\omega ^2dx=\frac{\beta(\omega)+1}{3\omega}\int \lvert \nabla Q_\omega \rvert^2dx,\quad
\int Q_\omega ^6dx=\beta(\omega)\int \lvert \nabla Q_\omega \rvert^2dx.
\end{split}
\end{equation*}
\end{proposition}

\begin{proposition}\label{p2.1}
For arbitrary $\varphi_0\in H^1(\mathbb{R}^3)$, (\ref{1.1}) possesses a unique global solution $\varphi\in C(\mathbb{R}; H^1(\mathbb{R}^3))$ such that $\varphi(0,x)=\varphi_0$. In addition, the solution holds the conservation of mass, energy and momentum, where mass is given by (\ref{1.29}), energy is given by (\ref{1.7}) and momentum is given by $P(u):=\int2 \mathrm{Im}(\bar{u}\nabla u)dx$ for $u\in H^1(\mathbb{R}^3)$.
\end{proposition}


\begin{proposition}\label{p2.16}

For the variational problem (\ref{1.31}), one has that:\\
(I) $d_{m}^{I}=\infty$ when $0<m<\frac{4}{3\sqrt{3}}\rho$.\\
(II) $0<d_{m}^{I}<\infty$ when $\frac{4}{3\sqrt{3}}\rho\leq m<\rho$.\\
(III) $d_{m}^{I}=d_{m}$ when $\rho\leq m<\infty$.\\
(IV) For $m\geq \frac{4}{3\sqrt{3}}\rho$, the infimum $d_{m}^{I}$ is achieved and is both strictly decreasing and lower semicontinuous as a function of $m$. In addition, $d_{m}^{I}$ is achieved at some positive minimizer $\psi$ when $m\geq \frac{4}{3\sqrt{3}}\rho$.
\end{proposition}

\begin{proposition}\label{p2.18}

For $\omega\in (0, \frac{3}{16})$, let
\begin{equation*}
R_{\omega}(x)=\sqrt{\frac{1+\beta(\omega)}{4\beta(\omega)}}Q_{\omega}(\frac{3[1+\beta(\omega)]}{4\sqrt{3\beta(\omega)}}x),
\end{equation*}
where $\beta(\omega)=\int Q_{\omega}^{6}dx/\int\lvert\nabla Q_{\omega}\lvert^{2}dx$. Then one has that
\begin{equation*}
I(R_{\omega})=0,\;\;\; M(R_{\omega})=\frac{16\sqrt{3\beta(\omega)}}{9[1+\beta(\omega)]^{2}}M(Q_{\omega}), \;\;\; M(R_{\omega})\leq M(Q_{\omega}).
\end{equation*}
Additionally $M(R_{\omega})= M(Q_{\omega})$ if and only if $\beta(\omega)=\frac{1}{3}$. Here $R_{\omega}$ is called rescaled soliton.
\end{proposition}

\begin{proposition}\label{p2.19}

Let $m\geq \frac{4}{3\sqrt{3}}\rho$ and $u\in H_{x}^{1}(\mathbb{R}^{3})\backslash\{0\}$ obey
\begin{equation*}
M(u)=m,\;\;\; E(u)=d_{m}^{I}, \;\;\; I(u)=0,
\end{equation*}
where $H_{x}^{1}(\mathbb{R}^{3}):=\{u\in H^{1}(\mathbb{R}^{3}), u(x)=u(\lvert x\lvert)\}$. Then either $u(x)=e^{i\theta}R_{\omega}(x+x_{0})$ or $u(x)=e^{i\theta}Q_{\omega}(x+x_{0})$ for some $\theta\in [0, 2\pi)$, some $x_{0}\in \mathbb{R}^{3}$ and some $0<\omega<\frac{3}{16}$ that obeys $\beta(\omega)>\frac{1}{3}$.
\end{proposition}

\begin{lemma}\label{p2.17}

For $\omega\in (0,  \frac{3}{16})$,  one has that
\begin{equation*}
\frac{4}{3\sqrt{3}}\rho \leq M(R_{\omega})\leq M(Q_{\omega}).
\end{equation*}
Moreover there exists
\begin{equation*}
\frac{4}{3\sqrt{3}}\rho \leq m_{0}\leq \rho
\end{equation*}
such that $M(Q_{\omega})\geq m_{0}$ for all $\omega\in (0, \frac{3}{16})$.
\end{lemma}

\begin{proof}
For $\omega\in (0,  \frac{3}{16})$,  let $M(R_{\omega})=m$.  From Proposition  \ref{p2.18} one has that
\begin{equation}\label{2.1}
R_{\omega}\in H^{1}(\mathbb{R}^{3}),\;\;\; E(R_{\omega})<\infty,\;\;\; M(R_{\omega})=m,\;\;\; I(R_{\omega})=0.
\end{equation}
Therefore Proposition \ref{p2.16} yields that $d_{m}^{I}<\infty$,  and
\begin{equation}\label{2.2}
m=M(R_{\omega})\geq \frac{4}{3\sqrt{3}}\rho.
\end{equation}
Thus Proposition  \ref{p2.18} implies that for all $\omega\in (0,  \frac{3}{16})$,
\begin{equation}\label{2.3}
M(Q_{\omega})\geq M(R_{\omega})\geq \frac{4}{3\sqrt{3}}\rho .
\end{equation}
Let
\begin{equation}\label{2.4}
m_{0}=\inf_{\{0<\omega<\frac{3}{16}\}}M(Q_{\omega}).
\end{equation}
From Proposition  \ref{p2.3},  $\phi$  satisfies (\ref{1.2}) with $\omega=\omega_{\phi}\in (0,  \frac{3}{16})$ and $M(\phi)=\rho$.  It follows that
\begin{equation}\label{2.5}
m_{0}\leq \rho.
\end{equation}
Then by  (\ref{2.3}),   (\ref{2.4}) and  (\ref{2.5}), one concludes that
\begin{equation}\label{2.6}
\frac{4}{3\sqrt{3}}\rho \leq m_{0}\leq \rho.
\end{equation}
This proves Lemma  \ref{p2.17}.
\end{proof}


\section{The cross-constrained variational problem}\label{sec3}
From Hmidi and Keraani \cite{HK2005}, we have the following profile decomposition theory.
\begin{proposition}\label{p3.1}
Let $\{u_n\}_{n=1}^\infty$ be a bounded sequence in $H^1(\mathbb{R}^3)$. Then there exists a subsequence of $\{u_n\}_{n=1}^\infty$, still denoted by $\{u_n\}_{n=1}^\infty$, a sequence
$\{U^j\}_{j=1}^\infty$ in $H^1(\mathbb{R}^3)$ and a family of
$\{x_n^j\}_{j=1}^\infty \subset\mathbb{R}^3$ satisfying the following properties.\\
(I) For every $j\neq k$, $\lvert x_n^j-x^k_n\rvert\rightarrow\infty$ as $n\rightarrow\infty$.\\
(II) For every $l\geq1$ and every $x\in\mathbb{R}^3$, $u_n(x)$ can be decomposed by
\begin{equation*}
u_n(x)=\sum_{j=1}^l U^j(x-x^j_n)+u^l_n,
\end{equation*}
with the remaining term $u^l_n=u^l_n(x)$ satisfying
\begin{equation*}
\lim_{l\rightarrow\infty}\limsup_{n\rightarrow\infty} \lvert\lvert u^l_n\rvert \rvert_{L^q(\mathbb{R}^3)}=0,\  \ for\ every \ q\in [2,6].
\end{equation*}
Moreover, as $n\rightarrow\infty$,
\begin{equation*}
\begin{split}
&\lvert\lvert u_n\rvert \rvert^2_{L^2}=\sum^l_{j=1}\lvert\lvert U^j \rvert \rvert^2_{L^2}
+\lvert\lvert u^l_n \rvert \rvert^2_{L^2}+o(1),\\
&\lvert\lvert \nabla u_n\rvert \rvert^2_{L^2}=\sum^l_{j=1}\lvert\lvert \nabla U^j \rvert \rvert^2_{L^2}
+\lvert\lvert \nabla u^l_n \rvert \rvert^2_{L^2}+o(1),\\
&\int[\sum^l_{j=1}U^j(x-x_n^j)]^{q}dx=\sum^l_{j=1}\int[U^j(x-x_n^j)]^{q}dx+o(1),
\end{split}
\end{equation*}
where $\lim\limits_{n\rightarrow\infty}o(1)=0$, $2\leq q\leq 6$.
\end{proposition}

The sequence $\{x_n^j\}_{n=1}^\infty$ in Proposition \ref{p3.1}  is called to satisfy the orthogonality condition if and only if for every $k\neq j$,\; $\lvert x_n^k-x_n^j\rvert\rightarrow\infty$ as \; $n\rightarrow\infty$.

By a direct calculation, we have the following proposition.

\begin{proposition}\label{p3.2}
For $u\in H^1(\mathbb{R}^{3})\backslash \{0\}$, let $u_{\lambda}=\lambda^{\frac{3}{2}}u(\lambda x)$ with $\lambda>0$. Then
\begin{equation*}
\int \rvert u_{\lambda}\rvert^{2}dx=\int \rvert u \rvert^{2}dx,\;\;\; \frac{d}{d\lambda}E(u_{\lambda})=\frac{3}{\lambda}I(u_{\lambda}).
\end{equation*}
In addition, $E(u_{\lambda})$ achieves the minimum at $\lambda_{0}$ when $\lambda_{0}$ satisfies $I(u_{\lambda_0})=0$.
\end{proposition}

Let $m\geq \frac{4}{3\sqrt{3}}\rho$ and $\psi$ be a positive minimizer of (\ref{1.31}). Then $\psi$ satisfies the Euler-Lagrange equation
\begin{equation}\label{3.1}
-\Delta u+\rvert u \rvert^{4}u-\rvert u \rvert^{2}u+\mu(-2\Delta u+6\rvert u \rvert^{4}u-3\rvert u \rvert^{2}u)+2\nu u=0,
\end{equation}
where $\mu$ and $\nu$ are the corresponding Lagrange multipliers. By \cite{KOPV2017}, one has that $\nu>0$, $\mu\geq 0$ and under the change of variables $u(x)=a v(\lambda x)$ with
\begin{equation}\label{3.2}
a^{2}=\frac{1+3\mu}{1+6\mu},\;\; \lambda^{2}=\frac{(1+3\mu)^{2}}{(1+2\mu)(1+6\mu)},\;\;\; \omega=\frac{2\nu(1+6\mu)}{(1+3\mu)^{2}},
\end{equation}
the Euler-Lagrange equation (\ref{3.1}) becomes (\ref{1.2}).

Now we define the following variational problem with the above $\psi$ and $\nu$
\begin{equation}\label{3.3}
d:=\mathop{\inf}\limits_ {\{u\in H^1(\mathbb{R}^{3}),\  0<\int\lvert u\rvert^2dx\leq\int
\lvert \psi\rvert^2dx,\  I(u)=0\}
}
[E(u)+\nu\int\lvert u\rvert^2dx].
\end{equation}
We see that (\ref{3.3}) is related to $d_{m}^{I}$ and $E(u)+\nu\int\lvert u\rvert^2dx$ is constrained to the flow $\{u\in H^{1}(\mathbb{R}^{3}), 0<\int\lvert u\rvert^2dx\leq\int
\lvert \psi\rvert^2dx\}$ with $I(u)=0$. Therefore we call (\ref{3.3}) cross-constrained variational problem, which is a new type of variational problem introduced in this paper. We have the following result.

\begin{theorem}\label{t3.3}
Let $m\geq \frac{4}{3\sqrt{3}}\rho$. Then the infimum $d$ in (\ref{3.3}) is achieved. Moreover $\psi$ is the unique positive minimizer of (\ref{3.3}) up to translations.
\end{theorem}
\begin{proof}
Since $m\geq \frac{4}{3\sqrt{3}}\rho$, from Proposition \ref{p2.16}, the variational problem (\ref{1.31}) is achieved at some positive minimizer $\psi$. Thus we can take a positive energy minimizer $\psi$ of $d_{m}^{I}$ and the corresponding Lagrange multiplier $2\nu$ of $\psi$. In addition $\psi$ satisfies that $I(\psi)=0$. Therefore
\begin{equation}\label{3.4}
\psi\in \mathcal{A}:=\{u\in H^1(\mathbb{R}^3), 0<\int  \lvert u\rvert^2dx\leq \int \lvert \psi \rvert^2dx, I(u)=0\},
\end{equation}
which leads to $\mathcal{A}$ is not empty.

For $u\in H^1(\mathbb{R}^3)$ and arbitrary $\varepsilon>0$ from Proposition \ref{p2.3}, one has the interpolation inequality
\begin{equation}\label{3.5}
\int\lvert u\rvert^4dx\leq\varepsilon\int\lvert u\rvert^6dx+\varepsilon^{-1}\int\lvert u\rvert^2dx.
\end{equation}
Let $u\in \mathcal{A}$. By (\ref{3.4}), $u\in H^1(\mathbb{R}^3)$ and $\int  \lvert u\rvert^2dx\leq \int \lvert \psi \rvert^2dx$.
From (\ref{1.7}) and (\ref{3.5})
\begin{equation}\label{3.6}
E(u)\geq\frac{1}{2}\int\lvert \nabla u\rvert^2dx-\frac{1}{4}\varepsilon^{-1}\int\lvert  u\rvert^2dx+(\frac{1}{6}-\frac{\varepsilon}{4})\int\lvert u\rvert^6dx.
\end{equation}
Taking $0<\varepsilon<\frac{2}{3}$, it follows that
\begin{equation}\label{3.7}
E(u)\geq-\frac{1}{4}\varepsilon^{-1}\int\lvert u\rvert^2dx\geq-\frac{1}{4}\varepsilon^{-1}\int\lvert  \psi\rvert^2dx=constant>-\infty.
\end{equation}
Therefore we have that for $u\in \mathcal{A}$
\begin{equation}\label{3.8}
H(u):=E(u)+\nu\int\lvert u\rvert^2dx\geq constant>-\infty.
\end{equation}
By (\ref{3.3}) we  deduce that $d>-\infty$.

Let $\{u_n\}_{n=1}^{\infty}\subset H^1(\mathbb{R}^3)$ be a minimizing  sequence of (\ref{3.3}). Then for all $n\in\mathbb{N}$,
\begin{equation}\label{3.9}
I(u_n)=0,\ 0<\int\lvert u_n\rvert^2dx\leq\int\psi^2dx, \ \
H(u_n)\rightarrow d,\quad n\rightarrow\infty.
\end{equation}
Thus for $n$ large enough we have that
\begin{equation}\label{3.10}
\frac{1}{2}\int\lvert \nabla u_n\rvert^2+\frac{1}{6}\lvert  u_n\rvert^6+
\nu\lvert  u_n\rvert^2dx<d+1+\frac{1}{4}\int\lvert  u_n\rvert^4dx.
\end{equation}
By (\ref{3.5}) and taking $0<\varepsilon<\frac{2}{3}$, (\ref{3.10}) implies that
\begin{equation}\label{3.11}
\frac{1}{2}\int\lvert \nabla u_n\rvert^2+
\nu\lvert  u_n\rvert^2dx<d+1+\frac{1}{4}\varepsilon^{-1}\int\lvert  \psi\rvert^2dx.
\end{equation}
Thus we deduce that $\{u_n\}^{\infty}_{n=1}$ is bounded in $H^1(\mathbb{R}^3)$.
In the following we divide two cases to proceed.

Case 1: $d\geq 0$.

Now we apply Proposition \ref{p3.1} to the minimizing sequence $\{u_n\}_{n=1}^{\infty}$. Then there exists a subsequence still denoted by $\{u_n\}_{n=1}^{\infty}$ such that
\begin{equation}\label{3.12}
u_n(x)=\sum_{j=1}^{l}U_n^j(x)+u_n^l,
\end{equation}
where $U^j_n(x):=U^j(x-x^j_n)$ and $u^l_n:=u^l_n(x)$ satisfies
\begin{equation}\label{3.13}
\lim_{l\rightarrow\infty}\limsup_{n\rightarrow\infty}\lvert\lvert u^l_n\rvert \rvert_{L^q(\mathbb{R}^3)}=0\quad with \quad  q\in [2,6].
\end{equation}
Moreover, by Proposition \ref{p3.1}, we have the following estimations as $n\rightarrow\infty$:
\begin{equation}\label{3.14}
\lvert\lvert u_n\rvert \rvert^2_{L^2}=\sum^l_{j=1}\lvert\lvert U_n^j\rvert \rvert^2_{L^2}+\lvert\lvert u_n^l\rvert \rvert^2_{L^2}+o(1),
\end{equation}
\begin{equation}\label{3.15}
\lvert\lvert \nabla u_n\rvert \rvert^2_{L^2}=\sum^l_{j=1}\lvert\lvert\nabla U_n^j\rvert \rvert^2_{L^2}+\lvert\lvert\nabla u_n^l\rvert \rvert^2_{L^2}+o(1),
\end{equation}
\begin{equation}\label{3.16}
\lvert\lvert u_n\rvert \rvert^4_{L^4(\mathbb{R}^3)}=\sum^l_{j=1}\lvert\lvert U_n^j\rvert \rvert^4_{L^4(\mathbb{R}^3)}+\lvert\lvert u_n^l\rvert \rvert^4_{L^4(\mathbb{R}^3)}+o(1),
\end{equation}
\begin{equation}\label{3.17}
\lvert\lvert u_n\rvert \rvert^6_{L^6(\mathbb{R}^3)}=\sum^l_{j=1}\lvert\lvert U_n^j\rvert \rvert^6_{L^6(\mathbb{R}^3)}+\lvert\lvert u_n^l\rvert \rvert^6_{L^6(\mathbb{R}^3)}+o(1).
\end{equation}
By (\ref{3.8}) we have that
\begin{equation}\label{3.18}
H(u_n)=\sum^l_{j=1}H(U^j_n)+H(u^l_n)+o(1), \quad as \quad n\rightarrow\infty,
\end{equation}
\begin{equation}\label{3.19}
I(u_n)=\sum^l_{j=1}I(U^j_n)+I(u^l_n)+o(1), \quad as \quad n\rightarrow\infty.
\end{equation}
By (\ref{3.9}), (\ref{3.13}) and $I(u_n)=0$, for $n$ large enough it follows that
\begin{equation}\label{3.20}
\sum^l_{j=1}H(U^j_n)\leq d.
\end{equation}
\begin{equation}\label{3.21}
\sum^l_{j=1}I(U^j_n)\leq 0.
\end{equation}
Since $d\geq0$, from (\ref{3.20}) there exists certain $U^j_n$ denoted by $U$ such that
\begin{equation}\label{3.22}
H(U)\leq d.
\end{equation}
From (\ref{3.14})
\begin{equation}\label{3.23}
0<\int\lvert U\rvert^2dx\leq\int \psi^2dx.
\end{equation}

Now we again divide  two situations to proceed.

The first situation: $U$ satisfies that
\begin{equation}\label{3.24}
\int\frac{1}{3}\lvert \nabla U\rvert^2-\frac{1}{4}\lvert  U\rvert^4dx<0.
\end{equation}
In this situation, there exists $U_\lambda=\lambda^{\frac{3}{2}}U(\lambda x)$ with certain $\lambda>0$ such that
\begin{equation}\label{3.25}
I(U_\lambda)=\int\frac{1}{3}\lambda^2\lvert \nabla U\rvert^2+\frac{1}{3}\lambda^6\lvert U\rvert^6-\frac{1}{4}\lambda^3\lvert U\rvert^4=0.
\end{equation}
\begin{equation}\label{3.26}
\int\lvert  U_\lambda\rvert^2dx=\int\lvert  U \rvert^2dx \leq \int \psi^2dx.
\end{equation}
Thus by (\ref{3.3}) we have that
\begin{equation}\label{3.27}
H(U_\lambda)\geq d.
\end{equation}
On the other hand, by Proposition \ref{3.2} and (\ref{3.22}) we have that
\begin{equation}\label{3.28}
H(U_\lambda)\leq H(U)\leq d.
\end{equation}
Therefore (\ref{3.27}) and (\ref{3.28}) imply that
\begin{equation}\label{3.29}
H(U_\lambda)= H(U)= d.
\end{equation}
Noting that (\ref{3.25}) and (\ref{3.26}), we have that (\ref{3.3}) is achieved at $U_\lambda$.

The second situation: $U$ satisfies that
\begin{equation}\label{3.30}
\int\frac{1}{3}\lvert \nabla U\rvert^2-\frac{1}{4}\lvert  U\rvert^4dx\geq0.
\end{equation}
From (\ref{3.30}) we have that
\begin{equation}\label{3.31}
E(U)\geq\int\frac{1}{6}\lvert \nabla U\rvert^2+\frac{1}{6}\lvert  U\rvert^6dx>0.
\end{equation}
Now we put
$$V=\mathop{\sum}^l_{j=1}U^j_n(x)-U,$$
that is
$$\sum^l_{j=1}U^j_n(x)=U+V.$$
By (\ref{3.20}) we have that
\begin{equation}\label{3.32}
E(U)+E(V)+\nu\int\lvert U\rvert^2dx+\nu\int\lvert V\rvert^2dx\leq d.
\end{equation}
From (\ref{3.31}) it follows that
\begin{equation}\label{3.33}
H(V)=E(V)+\nu\int\lvert V\rvert^2dx\leq d.
\end{equation}
From (\ref{3.30}) one has that $I(U)>0$. From (\ref{3.21}) one has that
$$I(U)+I(V)\leq 0.$$
It follows that
$I(V)<0$. Thus we get that $V\not\equiv 0$ and there exists $V_\lambda=\lambda^{\frac{3}{2}}V(\lambda x)$ with $\lambda\geq1$ such that $I(V_\lambda)=0$. In addition,
\begin{equation}\label{3.33a}
0<\int\lvert V_\lambda\rvert^2dx=\int\lvert V\rvert^2dx\leq\int \psi^2dx.
\end{equation}
By (\ref{3.3}) we have that
\begin{equation}\label{3.34}
H(V_\lambda)\geq d.
\end{equation}
On the other hand, by Proposition \ref{p3.2} and (\ref{3.33})
\begin{equation}\label{3.35}
H(V_\lambda)\leq H(V)\leq d.
\end{equation}
Therefore (\ref{3.34}) and (\ref{3.35}) imply that
\begin{equation}\label{3.36}
H(V_\lambda)= H(V)= d.
\end{equation}
Thus we have that (\ref{3.3}) is achieved at $V_\lambda$.

Then we conclude that (\ref{3.3}) is achieved  for Case 1.

Case 2: \quad $d<0$.

 In this case, we use Schwartz symmetric rearrangement theory \cite{LL2000} and Strauss radial compactness lemma \cite{S1977} to solve (\ref{3.3}). For a minimizing sequence $\{u_n\}^{\infty}_{n=1}$ of (\ref{3.3}), let $u^*_n$ be the Schwartz symmetrization of $u_n$ for every $n\in\mathbb{N}$. Thus $\{u^*_n\}^{\infty}_{n=1}$ is also bounded in $H^1(\mathbb{R}^3)$. In addition, for all $n\in\mathbb{N}$,
\begin{equation}\label{3.37}
0<\int\lvert u^*_n\rvert^2dx=\int\lvert u_n\rvert^2dx\leq\int\lvert \psi\rvert^2dx,
\end{equation}
\begin{equation}\label{3.38}
I(u^*_n)\leq 0, \quad H(u^*_n)\leq H(u_n).
\end{equation}
Moreover there exists a weak convergence subsequence of $\{u^*_n\}^{\infty}_{n=1}$ still denoted by $\{u^*_n\}^{\infty}_{n=1}$ such that
\begin{equation}\label{3.39}
u^*_n \rightharpoonup v \quad in \quad H^1(\mathbb{R}^3).
\end{equation}
From Strauss radial compactness Lemma \cite{S1977}, it follows that
\begin{equation}\label{3.40}
u^*_n \rightarrow v  \quad in \quad L^4(\mathbb{R}^3).
\end{equation}
Applying the weak lower semi-continuity, by (\ref{3.37}), (\ref{3.38}), (\ref{3.39}) and (\ref{3.40}), we deduce that
\begin{equation}\label{3.41}
0<\int\lvert v\rvert^2dx \leq\int\psi^2dx, \quad I(v)\leq0,\;\;\;\; H(v)\leq d.
\end{equation}
Since $d<0$, (\ref{3.41}) implies that $H(v)\leq d<0$. Therefore we have that $v\not\equiv 0$. Thus we let $v_\lambda=\lambda^{\frac{3}{2}}v(\lambda x)$ with $\lambda>0$. By $I(v)\leq 0$, there exists certain $\lambda>0$ such that
\begin{equation}\label{3.41a}
 I(v_\lambda)=0.
\end{equation}
 In terms of Proposition  \ref{p3.2}, we have that
\begin{equation}\label{3.42}
 H(v_\lambda)\leq H(v).
\end{equation}
On the other hand, since
\begin{equation}\label{3.42a}
0<\int\lvert v_\lambda\rvert^2dx=\int\lvert v\rvert^2dx\leq\int\psi^2 dx,
\end{equation}
by (\ref{3.41a}) and (\ref{3.3}), one derives that $H(v_\lambda)\geq d$. By
(\ref{3.41}) and (\ref{3.42}), it follows that
\begin{equation}\label{3.43}
 d\leq H(v_\lambda)\leq H(v)\leq d.
\end{equation}
Then we get that $H(v_\lambda)=d$. Thus (\ref{3.3}) is  achieved at $v_\lambda$.

Therefore we conclude that (\ref{3.3}) is always achieved. Obviously if $\xi$ is a minimizer of (\ref{3.3}), then $\lvert \xi(x)\rvert$ is also a minimizer of (\ref{3.3}). Thus (\ref{3.3}) always possesses the minimizer $\xi(x)\geq 0$ and $\xi(x)\not\equiv 0$.

Now we claim that
\begin{equation}\label{3.44}
\int \xi^{2}dx=\int \psi^{2}dx=m.
\end{equation}
Otherwise, by (\ref{3.3}) we have that
\begin{equation}\label{3.45}
m':=\int \xi^{2}dx<\int \psi^{2}dx=m.
\end{equation}
Then we construct the variational problem
\begin{equation}\label{3.46}
d_{m'}^{I}=\inf_{\{u\in H^{1}(\mathbb{R}^{3}), \; M(u)=m',\; I(u)=0\}}E(u).
\end{equation}
Thus $\xi$ is a solution of (\ref{3.46}). Moreover $d_{m'}^{I}=d$. On the other hand, according to Proposition \ref{p2.16}, (\ref{3.45}) yields that $d_{m'}^{I} > d_{m}^{I}$. But we always have that $d_{m}^{I}\geq d$. The above facts lead to $d=d_{m'}^{I} > d_{m}^{I}\geq d$. This is a contradiction. Therefore we deduce that (\ref{3.44}) is true. Thus $\xi$ is also a minimizer of $d_{m}^{I}$, that is
\begin{equation}\label{3.47}
\int \xi^{2}dx=\int \psi^{2}dx,\;\; I(\xi)=I(\psi)=0,\;\;E(\xi)=E(\psi).
\end{equation}
(\ref{3.47}) yields that $\psi$ is also a minimizer of the infimum $d$ defined in (\ref{3.3}). By (\ref{3.3}), there exists a Lagrange multiplier $\Lambda\in\mathbb{R}$ such that $\xi$ satisfies the Euler-Lagrange equation $d H\lvert_{\xi}+\Lambda d I\lvert_{\xi}=0$, that is
\begin{equation}\label{3.48}
-\Delta \xi-\xi^3+\xi^5+2\nu\xi+\Lambda(-2\Delta \xi-3\xi^3+6\xi^5)=0.
\end{equation}
By (\ref{3.48}), we have the following Pohozaev identities
\begin{equation}\label{3.49}
\int(1+2\Lambda)\lvert \nabla \xi\rvert^2-(1+3\Lambda)\xi^4+(1+6\Lambda) \xi^6+2\nu\xi^2dx=0,
\end{equation}
\begin{equation}\label{3.50}
\int(1+2\Lambda)\lvert \nabla \xi\rvert^2-\frac{3}{2}(1+3\Lambda)\xi^4+(1+6\Lambda) \xi^6+6\nu\xi^2dx=0.
\end{equation}
Since $I(\xi)=0$, by (\ref{3.49}) and (\ref{3.50}), we have that
\begin{equation}\label{3.51}
\Lambda\int \xi^4-\frac{16}{3}\xi^6dx=0.
\end{equation}
On the other hand, by (\ref{3.1}) we also have the Pohozaev identities
\begin{equation}\label{3.52}
\int(1+2\mu)\lvert \nabla \psi\rvert^2-(1+3\mu) \psi^4+(1+6\mu) \psi^6+2\nu \psi^2dx=0,
\end{equation}
\begin{equation}\label{3.53}
\int(1+2\mu)\lvert \nabla \psi\rvert^2-\frac{3}{2}(1+3\mu)\psi^4+(1+6\mu) \psi^6+6\nu \psi^2dx=0.
\end{equation}
Since $I(\psi)=0$, by (\ref{3.52}) and (\ref{3.53}) we also have that
\begin{equation}\label{3.54}
\mu\int \psi^4-\frac{16}{3}\psi^6dx=0.
\end{equation}

If $\Lambda=0$, by (\ref{3.47}) and (\ref{3.48}), Proposition \ref{p2.19} asserts that
\begin{equation}\label{3.200}
\xi=e^{i\theta}Q_{\omega}(x+x_{0})
\end{equation}
for some $\theta\in [0, 2\pi)$, some $x_{0}\in \mathbb{R}^{3}$ and some $0<\omega<\frac{3}{16}$. At the same time, by (\ref{3.47}), (\ref{3.200}) and Proposition \ref{p2.19}, it also yields
\begin{equation}
\psi=e^{i\theta}Q_{\omega}(x+x_{0}).
\end{equation}
Thus by (\ref{3.1}) we get that $\mu=0$. In turn,  if $\mu=0$, by the same arguments one gets that $\Lambda=0$.

If $\Lambda\neq 0$, then $\mu\neq 0$. 
By (\ref{3.51}) we have
\begin{equation}\label{3.55}
\int \xi^4-\frac{16}{3}\xi^6dx=0.
\end{equation}
From $I(\xi)=0$, it follows that
\begin{equation}\label{3.56}
\int \xi^4dx=\frac{16}{3}\int \xi^6dx=\frac{16}{9}\int \lvert \nabla \xi \lvert^{2}dx .
\end{equation}
When $\mu\neq 0$, by (\ref{3.54}) we have
\begin{equation}\label{3.57}
\int \psi^4-\frac{16}{3}\psi^6dx=0.
\end{equation}
From $I(\psi)=0$, it follows that
\begin{equation}\label{3.58}
\int \psi^4dx=\frac{16}{3}\int \psi^6dx=\frac{16}{9}\int \lvert \nabla \psi \lvert^{2}dx .
\end{equation}
Thus by (\ref{3.47}), (\ref{3.49}), (\ref{3.50}), (\ref{3.52}), (\ref{3.53}), we get that
\begin{equation}\label{3.59}
\int \xi^4dx=\frac{1+3\Lambda}{1+3\mu}\int \psi^4dx ,
\end{equation}
\begin{equation}\label{3.60}
\int \xi^6dx=\frac{1+3\Lambda}{1+3\mu}\int \psi^6dx ,
\end{equation}
\begin{equation}\label{3.61}
\int \lvert \nabla \xi \lvert^{2}dx=\frac{1+3\Lambda}{1+3\mu}\int \lvert \nabla \psi \lvert^{2}dx .
\end{equation}
Since (\ref{3.47}), one has that $E(\xi)=E(\psi)$. Thus (\ref{3.59}), (\ref{3.60}) and (\ref{3.61}) yield that
\begin{equation}
\frac{1+3\Lambda}{1+3\mu}=1.
\end{equation}
Therefore we get that $\Lambda=\mu$. Thus we conclude that one always that $\Lambda=\mu$ whatever $\Lambda=0$ or $\Lambda\neq 0$. Now under the change of variables $u(x)=av(\lambda x)$ with (\ref{3.2}) for $(u, v)=(\psi, v)$ and $(u, v)=(\xi, v)$ respectively. Both (\ref{3.1}) and (\ref{3.48}) deduce the equation (\ref{1.37}) with the same
\begin{equation}\label{3.63}
\omega=\frac{2\nu(1+6\mu)}{(1+3\mu)^{2}}=\frac{2\nu(1+6\Lambda)}{(1+3\Lambda)^{2}}.
\end{equation}
Thus (\ref{1.37}) has a unique positive solution $Q_{\omega}$. Since both $\psi(x)=aQ_{\omega}(\lambda x)$ and $\xi(x)=aQ_{\omega}(\lambda x)$ through (\ref{3.2}), we get that for some $x_{0}\in\mathbb{R}^{3}$,
\begin{equation}
\xi(x)\equiv \psi(x+x_{0}).
\end{equation}
Since $\xi(x)\geq 0$ and $\xi(x)\not\equiv 0$, we get that $v(x)\geq 0$ and $v(x)\not\equiv 0$. Using the strong maximum principle, we get that $v(x)>0$ for all $x\in\mathbb{R}^{3}$, therefore we get that both $\xi(x)$ and $\psi(x)$ are positive.

Although uniqueness of $\psi$ in (\ref{1.8}) is uncertain, $\psi$ and $\nu$ used in the above proof are determined. From the above proof, for arbitrary minimizer $\xi$ of (\ref{3.3}) such that $\xi(x)\geq0$ and $\xi(x)\not\equiv 0$, one has that $\xi(x)\equiv\psi(x)$ up to a translation. Therefore we assert that $\psi$ is the unique positive minimizer of (\ref{3.3}) up to translations.

This completes the proof of Theorem \ref{t3.3}.
\end{proof}
\begin{remark}
In \cite{Z2005,GZ2008,Z2002}, in order to prove instability of solitons, we construct some cross-constrained variational problems, but we can not solve them. In this paper, according to \cite{Z2000}, we construct a new cross-constrained variational problem (\ref{3.3}). Moreover we develop a set of intricate methods to resolve (\ref{3.3}). In proof of Theorem \ref{t3.3}, the treatment of Case 1 and the computation of Euler-Lagrange equation are especially remarkable that thanks \cite{KOPV2017}.

\end{remark}

\section{Uniqueness of the energy minimizers}\label{sec4}

We apply Theorem \ref{t3.3} to prove the following uniqueness.
\begin{theorem}\label{t4.1}
Let $m\geq \frac{4}{3\sqrt{3}}\rho$. Then the positive minimizer of variational problem (\ref{1.31}) is unique up to translations.
\end{theorem}
\begin{proof}
Since $m\geq \frac{4}{3\sqrt{3}}\rho$, by Proposition \ref{p2.16} the  variational problem (\ref{1.31}) possesses a positive minimizer. Now suppose that $\psi_1$ and $\psi_2$ are two different positive minimizers of (\ref{1.31}) up to translations. Thus by (\ref{1.31}) one has that
\begin{equation}\label{4.1}
\int \psi_1^2dx=m=\int \psi_2^2dx,
\end{equation}
\begin{equation}\label{4.2}
E(\psi_1)=d_m=E(\psi_2),
\end{equation}
\begin{equation}\label{4.3}
I(\psi_1)=0=I(\psi_2).
\end{equation}
Moreover $\psi_{1}$ and $\psi_{2}$ satisfy the Euler-Lagrange equation (\ref{3.1}) with $(\mu, \nu)=(\mu_{1}, \nu_{1})$ and $(\mu, \nu)=(\mu_{2}, \nu_{2})$ respectively. Here $(\mu_{1}, \nu_{1})$ and $(\mu_{2}, \nu_{2})$ are the corresponding Lagrange multipliers for the minimizer $\psi_{1}$ and $\psi_{2}$ of $d_{m}^{I}$ respectively.
Thus from Theorem \ref{t3.3}, $\psi_1$ is the unique positive minimizer of the variational problem
\begin{equation}\label{4.4}
d_1:=\mathop{\inf}\limits_ {\{u\in H^1(\mathbb{R}^3),\  0<\int\lvert u\rvert^2dx\leq
\int \psi_1^2dx,\  I(u)=0\}
}
[E(u)+\nu_{1}\int\lvert u\rvert^2dx].
\end{equation}
By (\ref{4.1}) and (\ref{4.3}) one has that
\begin{equation}\label{4.5}
\psi_2\in\{u\in H^1(\mathbb{R}^3), \;\; 0<\int \lvert u\rvert^2dx\leq \int \psi_1^2dx,\;\; I(u)=0\}.
\end{equation}
From Theorem \ref{t3.3}, $\psi_1$ is unique. Therefore
\begin{equation}\label{4.6}
E(\psi_1)+\nu_1\int \psi_1^2dx<E(\psi_2)+\nu_1\int \psi_2^2dx.
\end{equation}
At the same time, $\psi_2$ is the unique positive minimizer of the variational problem
\begin{equation}\label{4.7}
d_2:=\mathop{\inf}\limits_ {\{u\in H^1(\mathbb{R}^3),\  0<\int\lvert u\rvert^2dx\leq
\int \psi_2^2dx,\  I(u)=0\}
}
[E(u)+\nu_{2}\int\lvert u\rvert^2dx].
\end{equation}
By (\ref{4.1}) and (\ref{4.3}),
\begin{equation}\label{4.8}
\psi_1\in\{u\in H^1(\mathbb{R}^3),\;\; 0<\int \lvert u\rvert^2dx\leq \int \psi_2^2dx, \;\;I(u)=0\}.
\end{equation}
Thus by the uniqueness of $\psi_2$, one implies that
\begin{equation}\label{4.9}
E(\psi_2)+\nu_2\int \psi_2^2dx<E(\psi_1)+\nu_2\int \psi_1^2dx.
\end{equation}
From (\ref{4.1}), (\ref{4.2}), (\ref{4.8}) and (\ref{4.9}), it yields that $\nu_1<\nu_1$
and $\nu_2<\nu_2$. This is impossible. Therefore one asserts that $\psi_1\equiv \psi_2$ up to a translation.

This proves Theorem \ref{t4.1}.
\end{proof}

In terms of Theorem \ref{t4.1}, we get more uniqueness results as follows.
\begin{theorem}\label{t4.13}
The positive minimizer of variational problem (\ref{1.8}) with $m\geq \rho$ is unique up to translations.
\end{theorem}
\begin{proof}
 Let $\psi_{m}$ ba a positive minimizer of (\ref{1.8}). Then $\psi_{m}$ satisfies (\ref{1.2}) with some $\omega\in (0, \frac{3}{16})$. Thus $I(\psi_{m})=0$. Therefore $\psi_{m}$ is a positive minimizer of the variational problem
\begin{equation*}
d_{m}^{I}=\inf_{\{u\in H^{1}(\mathbb{R}^{3}), \; M(u)=m,\; I(u)=0\}}E(u).
\end{equation*}
 By Theorem \ref{t4.1}, $d_{m}^{I}$ possesses unique positive minimizer. Then $\psi_{m}$ is unique.

This proves Theorem \ref{t4.13}
\end{proof}

\begin{theorem}\label{t4.2}
The positive minimizer of variational problem (\ref{1.5}) is unique up to translations.
\end{theorem}
\begin{proof}
By Proposition \ref{p2.3}, the variational problem (\ref{1.5}) is achieved at some positive minimizer $\phi$. Now let $\phi$ be arbitrary positive minimizer of (\ref{1.5}). Then $\phi$
satisfies
\begin{equation}\label{4.9a}
\int \phi^2dx=\frac{64}{9}d_0^2=\rho,\;\;\;\;  E(\phi)=0.
\end{equation}
From Proposition \ref{p2.5},
$\phi$ is a  positive minimizer of the variational problem
\begin{equation}\label{4.10}
d_\rho=\mathop{\inf}\limits_ {\{u\in H^1(\mathbb{R}^3),\  \int\lvert u\rvert^2dx=\rho\}
}
E(u).
\end{equation}
In terms of Proposition \ref{p2.16} and Theorem \ref{t4.13}, the positive minimizer $\psi_\rho$ of (\ref{4.10}) is unique up to translations. Therefore $\phi\equiv\psi_\rho$  up to a translation.
Thus we conclude that $\phi$ is unique up to translations.

This proves Theorem \ref{t4.2}.
\end{proof}

\begin{theorem}\label{t4.3}
Let $m\geq \frac{4}{3\sqrt{3}}\rho$ and $\psi$ be the unique positive minimizer of (\ref{1.31}) up to translations.
Then the set of all minimizers of (\ref{1.31}) is that
$$S_m=\{e^{i\theta}\psi(\cdot+y), \; \theta\in\mathbb{R},\; y\in\mathbb{R}^3\}.$$
\end{theorem}
\begin{proof}
By Proposition \ref{p2.16} and Theorem \ref{t4.1}, for $m\geq \frac{4}{3\sqrt{3}}\rho$, the variational problem
(\ref{1.31}) has a unique positive minimizer $\psi(x)$ up to a translation. Now assume that $\eta\in H^1(\mathbb{R}^3)$ is an arbitrary solution of (\ref{1.31}). Let $\eta=\psi^{1}+i\psi^{2}$, where $\psi^{1}, \ \psi^{2}\in H^1(\mathbb{R}^3)$ are real valued. Then $\tilde{\eta}=\lvert \psi^{1}\rvert+i\lvert \psi^{2}\rvert$
is still a solution of (\ref{1.31}). Thus there exist the Lagrange  multipliers $\mu\geq 0$, $\nu>0$ such that $\eta$ and $\tilde{\eta}$ satisfy (\ref{3.1}). Then for $j=1,2$,
\begin{equation}\label{4.11}
\begin{aligned}
-\Delta \psi^{j}&+\rvert \eta \rvert^{4}\psi^{j}-\rvert \eta \rvert^{2}\psi^{j}+\mu(-2\Delta \psi^{j}+6\rvert \eta \rvert^{4}\psi^{j}\\
&-3\rvert \eta \rvert^{2}\psi^{j})+2\nu \psi^{j}=0,
\end{aligned}
\end{equation}
\begin{equation}\label{4.12}
\begin{aligned}
-\Delta \rvert \psi^{j} \rvert & +\rvert \eta \rvert^{4}\rvert \psi^{j} \rvert -\rvert \eta \rvert^{2}\rvert \psi^{j}\rvert+\mu(-2\Delta \rvert \psi^{j} \rvert+6\rvert \eta \rvert^{4}\rvert \psi^{j} \rvert \\
&-3\rvert \eta \rvert^{2}\rvert \psi^{j}\rvert)+2\nu \rvert \psi^{j}\rvert=0.
\end{aligned}
\end{equation}
Under the change of variables $u(x)=av(\lambda x)$ with (\ref{3.2}) for $(u,  v)=(\psi^{j},  V^{j})$ and $(u,  v)=(\rvert \psi^{j}\rvert,  \rvert V^{j}\rvert)(j=1,2)$ respectively,  by (\ref{4.11}) and (\ref{4.12}),  one yields that
for $j=1,2$,
\begin{equation}\label{4.13}
\Delta V^j+\lvert U\rvert^2V^j-\lvert U\rvert^4V^j=\omega V^j \quad in \quad \mathbb{R}^3,
\end{equation}
\begin{equation}\label{4.14}
\Delta \lvert V^j\rvert+\lvert U\rvert^2\lvert V^j\rvert-\lvert U\rvert^4\lvert V^j\rvert=\omega \lvert V^j\rvert \quad in \quad \mathbb{R}^3,
\end{equation}
where $U=V^{1}+i V^{2}$. This shows that $\omega$ is the first eigenvalue of the operator $\Delta+\lvert U\rvert^2-\lvert U\rvert^4$ acting over $H^1(\mathbb{R}^3)$ and thus $V^1,\; V^2,\; \lvert V^1\rvert$ and $\lvert V^2\rvert$ are all multipliers of a positive normalized eigenfunction $v_0$ of $\Delta+\lvert U\rvert^2-\lvert U\rvert^4$, i.e.
\begin{equation}\label{4.15}
\Delta v_0+\lvert U\rvert^2v_0-\lvert U\rvert^4v_0=\omega v_0 \quad in \quad \mathbb{R}^3,
\end{equation}
with
\begin{equation}\label{4.16}
v_0\in C^2(\mathbb{R}^3)\cap H^1(\mathbb{R}^3), \ v_0>0 \ in \ \mathbb{R}^3.
\end{equation}
It is now obvious to deduce that $U=e^{i\theta}v_0$ for some $\theta\in\mathbb{R}$ and that $v_0$ is still a solution of (\ref{1.37}). Now we conduct the change of variables $u(x)=av(\lambda x)$ with (\ref{3.2}) for $(u, v)=(\xi, v_{0})$. Then $\xi$ satisfies (\ref{3.1}) with the same $\mu$ and $\nu$. Thus by Theorem \ref{t4.1},  we have that $\eta=e^{i\theta}\xi$ and $\xi=\psi(\cdot+y)$ for some $y\in \mathbb{R}^{3}$. Therefore one gets that $\eta=e^{i\theta}\psi(\cdot+y)$ for some $\theta\in\mathbb{R}$ and $y\in \mathbb{R}^{3}$.
It is obvious that for any $\theta\in\mathbb{R}$ and $y\in\mathbb{R}^3$, \; $e^{i\theta}\psi(\cdot+y)$ is also a solution of (\ref{1.31}). Therefore
\begin{equation}\label{4.17}
S_m=\{e^{i\theta}\psi(\cdot+y), \ \theta\in\mathbb{R},\ y\in\mathbb{R}^3\}
\end{equation}
is the set of all solutions of (\ref{1.31}).

This completes the proof of Theorem \ref{t4.3}.
\end{proof}

\begin{theorem}\label{t4.14}
Let $M(Q_{\omega_{1}})=m_{1}$ for some $\omega_{1}\in (0, \frac{3}{16})$. Denote $M(R_{\omega_{1}})=m$. Then the variational problem
\begin{equation*}
d_{m}^{I}:=\inf_{\{u\in H^{1}(\mathbb{R}^{3}), \;  M(u)=m, \; I(u)=0\}}E(u)
\end{equation*}
possesses a unique positive minimizer $\psi$ up to a translation. Let $Q_{\omega}$ be the unique rescaling function by $\psi(x)=aQ_{\omega}(\lambda x)$ with (\ref{3.2}). Then $Q_{\omega}$ satisfies (\ref{1.2}) with $\omega\in(0, \frac{3}{16})$ determined uniquely by $m$. Moreover $M(Q_{\omega})=m_{1}$.
\end{theorem}
\begin{proof}
By Lemma \ref{p2.17}, we have that $m=M(R_{\omega_{1}})\geq\frac{4}{3\sqrt{3}}\rho$.
By Theorem \ref{t4.1}, one concludes that $d_{m}^{I}$ is achieved at a unique positive minimizer $\psi$ up to a translation. Thus $\psi$ satisfies the Euler-Lagrange equation (\ref{3.1}) and
\begin{equation}
M(\psi)=m=M(R_{\omega_{1}})=\frac{16\sqrt{3\beta(\omega_{1})}}{9[1+\beta(\omega_{1})]^{2}}m_{1}.
\end{equation}
Under the change of variables $\psi=av(\lambda x)$ with (\ref{3.2}), thus the Euler-lagrange equation (\ref{3.2}) becomes (\ref{1.37}). Therefore one concludes that $v=Q_{\omega}$ for some $\omega\in (0, \frac{3}{16})$, which is determined by $m$. Since $\beta(\omega_{1})=\frac{1+6\mu}{3(1+2\mu)}$, it follows that
\begin{equation*}
\begin{aligned}
M(Q_{\omega})=&\frac{\lambda^{3}}{a^{2}}M(\psi)=\frac{(1+3\mu)\sqrt{1+6\mu}}{(1+2\mu)^{\frac{3}{2}}}m\\
=&\frac{(1+3\mu)\sqrt{1+6\mu}}{(1+2\mu)^{\frac{3}{2}}}\cdot\frac{16\sqrt{3\frac{1+6\mu}{3(1+2\mu)}}}{9[1+\frac{1+6\mu}{3(1+2\mu)}]^{2}}m_{1}=m_{1}.
\end{aligned}
\end{equation*}

This proves Theorem \ref{t4.14}.
\end{proof}


\begin{theorem}\label{t4.25}
For $m_{0}$ defined by (\ref{1.120}),  there exists unique $\omega_{\ast}\in (0, \frac{3}{16})$ such that $M(Q_{\omega_{\ast}})=m_{0}$.
\end{theorem}

\begin{proof}
By Lemma \ref{p2.17},
\begin{equation}
m_{0}=\inf_{\{0<\omega<\frac{3}{16}\}}M(Q_{\omega}),\;\;\; \frac{4}{3\sqrt{3}}\rho\leq m_{0}\leq\rho.
\end{equation}
Then Proposition \ref{p2.2a} implies that there exists $\omega_{0}\in (0, \frac{3}{16})$ such that $M(Q_{\omega_{0}})=m_{0}$.  Let $M(R_{\omega_{0}})=m$.  By Lemma \ref{p2.17},
$m\geq\frac{4}{3\sqrt{3}}\rho$. According to Proposition \ref{p2.16} and Theorem \ref{t4.1},  the variational  problem $d_{m}^{I}$ possesses a unique positive minimizer $\psi$.  Let $Q_{\omega_{\ast}}$  be the unique rescaling function by
\begin{equation}
\psi(x)=aQ_{\omega_{\ast}}(\lambda x)
\end{equation}
with (\ref{3.2}).  Then $Q_{\omega_{\ast}}$ satisfies (\ref{1.2}) and $\omega_{\ast}$ is uniquely determined by $m_{0}$.  Moreover from Theorem \ref{t4.14},
\begin{equation}
M(Q_{\omega_{\ast}})=m_{0}.
\end{equation}
This proves Theorem \ref{t4.25}.
\end{proof}

For $m\geq\frac{4}{3\sqrt{3}}\rho$, the variational problem (\ref{1.31}) possesses a unique positive minimizer $\psi$ up to a translation. We call this minimizer $\psi$ the energy minimizer. Then $\psi$ satisfies the Euler-Lagrange equation (\ref{3.1}) with the Lagrange multipliers $\mu\geq 0$ and $\nu>0$. Under the change of variables $\psi(x)=av(\lambda x)$ with (\ref{3.2}), one uniquely determines $v(x)$ satisfying (\ref{1.37}). We call this $v(x)$ the energy ground state. Thus $v=Q_{\omega}$ up to a translation for certain $\omega\in (0, \frac{3}{16})$. In fact, by (\ref{3.2})
\begin{equation}
\omega=\frac{2\nu(1+6\mu)}{(1+3\mu)^{2}}.
\end{equation}
Then for every $m\geq\frac{4}{3\sqrt{3}}\rho$, we can get the above $\omega\in (0, \frac{3}{16})$. We denote the set including all these $\omega$ by $\mathcal{B}$.

%

\begin{theorem}\label{t4.16}
$\mathcal{B}\subset (0, \frac{3}{16})$ and $\omega_{\ast}\in \mathcal{B}$.
\end{theorem}
\begin{proof}
It is clear that $\mathcal{B}\subset (0, \frac{3}{16})$. According to Theorem \ref{t4.14} and Theorem \ref{t4.25}, $\omega_{\ast}$ satisfies that
\begin{equation}
M(Q_{\omega_{\ast}})=m_{0},\;\;\; M(R_{\omega_{\ast}})=m\geq\frac{4}{3\sqrt{3}}\rho.
\end{equation}
Thus $d_{m}^{I}$ uniquely determines the energy minimizer $\psi$ and this $\psi$ uniquely determines the energy ground state $Q_{\omega_{\ast}}$. Therefore $\omega_{\ast}\in \mathcal{B}$.

This proves Theorem \ref{t4.16}.
\end{proof}

Now we complete the proof of Theorem \ref{t1.1}.

Proof of Theorem \ref{t1.1}: By Theorem \ref{t3.3}, Theorem \ref{t4.1}, Theorem \ref{t4.13}, Theorem \ref{t4.2} and Theorem \ref{t4.25}, one deduces Theorem \ref{t1.1}.

\section{Stability of the energy ground states}\label{sec10}
\begin{theorem}\label{t5.1}
Let $m\geq \frac{4}{3\sqrt{3}}\rho$. Then for arbitrary minimizing sequence $\{u_n\}^{\infty}_{n=1}$ of (\ref{1.31}), there exists a subsequence still denoted by  $\{u_n\}^{\infty}_{n=1}$  such that for $n$ large enough
$$\int \lvert u_n\rvert^4dx\geq c>0, \quad
\int \lvert \nabla u_n\rvert^2dx\geq c>0.$$
\end{theorem}
\begin{proof}
Since $m\geq \frac{4}{3\sqrt{3}}\rho$, by Proposition \ref{p2.16}, the variational problem (\ref{1.31}) is achieved. Take arbitrary minimizing sequence $\{u_n\}^{\infty}_{n=1}$ of (\ref{1.31}) such that
\begin{equation}\label{5.0a}
M(u_n)=\int \lvert u_n\rvert^2dx=m,\;\;I(u_{n})=0,\;\; E(u_n)\rightarrow d_m^{I} \ as \ n\rightarrow\infty.
\end{equation}
It is clear that $\{u_n\}^{\infty}_{n=1}$ is bounded in $H^1(\mathbb{R}^3)$. Therefore there exist $v\in H^1(\mathbb{R}^3)$ and certain subsequence of $\{u_n\}^{\infty}_{n=1}$, still denoted by $\{u_n\}^{\infty}_{n=1}$ such that
\begin{equation}\label{5.0b}
u_n \rightharpoonup v \ in \ H^1(\mathbb{R}^3).
\end{equation}
In the following we divide three cases to proceed.

Case 1: $d_m^{I}=0$.

In this case, $m=\rho$. From (\ref{5.0b}) we have that
\begin{equation}\label{5.0cc}
\int \lvert v\rvert^2dx\leq\rho.
\end{equation}
If
\begin{equation}\label{5.0c}
\int \lvert v\rvert^2dx=\rho,
\end{equation}
we directly get that $v\not\equiv 0$. If
\begin{equation}\label{5.0d}
\int \lvert v\rvert^2dx<\rho,
\end{equation}
 from Proposition \ref{p2.5} we have that $E(v)>0$. It follows that $v\not\equiv 0$.
Therefore for (\ref{5.0cc}) we always get $v\not\equiv 0$. Thus by the weak lower semi-continuity, for $n$ large enough we have that
\begin{equation}\label{5.4aa}
\int \lvert  u_n\rvert^4dx\geq \int \lvert v\rvert^4dx=c>0,
\end{equation}

\begin{equation}\label{5.0e}
\int \lvert \nabla u_n\rvert^2dx\geq \int \lvert \nabla  v\rvert^2dx=c>0.
\end{equation}

Case 2: $d_m^{I}<0$.

In this case, $m>\rho$ and $d_{m}^{I}=d_{m}$. Since $d_m<0$, for $n$ large enough we can choose a positive number $\delta$ such that
\begin{equation}\label{5.0f}
0<\delta<-d_m \ and \ E(u_n)\leq d_m+\delta<0.
\end{equation}
It follows that for $n$ large enough
\begin{equation}\label{5.1}
\int\lvert u_n\rvert^4dx\geq-4(d_m+\delta)=c>0.
\end{equation}
Applying Nirenberg-Gagliardo inequality
\begin{equation}\label{5.2}
\lvert\lvert u\rvert \rvert^4_{L^4(\mathbb{R}^3)}\leq  c \lvert\lvert \nabla u\rvert \rvert^3_{L^2} \lvert\lvert  u\rvert \rvert_{L^2},
\end{equation}
noting that $\lvert\lvert  u_n\rvert \rvert^2_{L^2}=m$, (\ref{5.1}) implies that
\begin{equation}\label{5.2a}
\int \lvert \nabla u_n\rvert^2dx \geq c>0.
\end{equation}

Case 3: $d_m^{I}>0$.

In this case, $m\geq\frac{4}{3\sqrt{3}}\rho$. Since $d_{m}^{I}>0$, for $n$ large enough we can choose a positive $\delta$ such that
\begin{equation}
E(u_{n})\geq d_{m}^{I}-\delta>0.
\end{equation}
By $I(u_{n})=0$ one gets
\begin{equation}
E(u_{n})=\int\frac{1}{6}\rvert \nabla u_{n}\rvert^{2}-\frac{1}{6}\rvert u_{n}\rvert^{6}dx.
\end{equation}
It follows that
\begin{equation}
\int\rvert \nabla u_{n}\rvert^{2}dx\geq c>0
\end{equation}
for $n$ large enough. On the other hand, $I(u_{n})=0$ derives that
\begin{equation}
\int\frac{1}{4}\rvert u_{n}\rvert^{4}dx=\int\frac{1}{3}\rvert \nabla u_{n}\rvert^{2}+\frac{1}{3}\rvert u_{n}\rvert^{6}dx.
\end{equation}
Thus one gets that
\begin{equation}
\int\rvert u_{n}\rvert^{4}dx\geq c>0
\end{equation}
for $n$ large enough.

The proof of Theorem \ref{t5.1} is completed.

\end{proof}
\begin{theorem}\label{t5.2}
Let $m\geq \frac{4}{3\sqrt{3}}\rho$ and $\psi$ be the unique positive minimizer of (\ref{1.31}). Then for arbitrary minimizing sequence $\{u_n\}^{\infty}_{n=1}$ of (\ref{1.31}), there exists a subsequence still denoted by  $\{u_n\}^{\infty}_{n=1}$  such that for some $\theta\in\mathbb{R}$ and $y\in\mathbb{R}^3$
$$u_n\rightarrow \psi(.+y)e^{i\theta} \ in \ H^1(\mathbb{R}^3), \ as\  n\rightarrow\infty.$$
\end{theorem}

\begin{proof}
Since $m\geq \frac{4}{3\sqrt{3}}\rho$, by Proposition \ref{p2.16} and Theorem \ref{t4.1}, the variational problem (\ref{1.31}) possesses a unique positive minimizer $\psi(x)$ up to translations. Now we apply Proposition \ref{p3.1} to the minimizing sequence $\{u_n\}^{\infty}_{n=1}$ of (\ref{1.31}). Then there exists a subsequence still denoted by $\{u_n\}^{\infty}_{n=1}$ such that
\begin{equation}\label{5.3}
u_n(x)=\sum_{j=1}^{l}U_n^j(x)+u_n^l,
\end{equation}
where $U^j_n(x):=U^j(x-x^j_n)$ and $u^l_n:=u^l_n(x)$ satisfies
\begin{equation}\label{5.4}
\lim_{l\rightarrow\infty}\limsup_{n\rightarrow\infty}\lvert\lvert u^l_n\rvert \rvert_{L^q(\mathbb{R}^3)}=0\quad with \quad  q\in [2,6],
\end{equation}
\begin{equation}\label{5.5}
\lvert\lvert u_n\rvert \rvert^2_{L^2}=\sum^l_{j=1}\lvert\lvert U_n^j\rvert \rvert^2_{L^2}+\lvert\lvert u_n^l\rvert \rvert^2_{L^2}+o(1),
\end{equation}
\begin{equation}\label{5.6}
\lvert\lvert \nabla u_n\rvert \rvert^2_{L^2}=\sum^l_{j=1}\lvert\lvert\nabla U_n^j\rvert \rvert^2_{L^2}+\lvert\lvert\nabla u_n^l\rvert \rvert^2_{L^2}+o(1),
\end{equation}
\begin{equation}\label{5.7}
\lvert\lvert u_n\rvert \rvert^4_{L^4(\mathbb{R}^3)}=\sum^l_{j=1}\lvert\lvert U_n^j\rvert \rvert^4_{L^4(\mathbb{R}^3)}+\lvert\lvert u_n^l\rvert \rvert^4_{L^4(\mathbb{R}^3)}+o(1),
\end{equation}
\begin{equation}\label{5.8}
\lvert\lvert u_n\rvert \rvert^6_{L^6(\mathbb{R}^3)}=\sum^l_{j=1}\lvert\lvert U_n^j\rvert \rvert^6_{L^6(\mathbb{R}^3)}+\lvert\lvert u_n^l\rvert \rvert^6_{L^6(\mathbb{R}^3)}+o(1),
\end{equation}
Thus we have that
\begin{equation}\label{5.9}
E(u_n)=\sum^l_{j=1}E(U^j_n)+E(u^l_n)+o(1) \quad as \quad n\rightarrow \infty.
\end{equation}
For $j=1,...,l$, let
\begin{equation}\label{5.10}
\lambda_j=\left(\frac{\lvert\lvert u_n\rvert \rvert_{L^2}}{\lvert\lvert U^j_n\rvert \rvert_{L^2}}\right)^{\frac{2}{3}},\quad
\lambda^l_n=\left(\frac{\lvert\lvert u_n\rvert \rvert_{L^2}}{\lvert\lvert u^l_n\rvert \rvert_{L^2}}\right)^{\frac{2}{3}}.
\end{equation}
By (\ref{5.5}) we have that $\lambda_j\geq 1$ and $\lambda^l_n\geq1$. In addition, from the convergence of $\mathop{\sum}^l\limits_{j=1}\lvert\lvert U^j_n\rvert \rvert^2_{L^2}$, there exists a $j_0\geq 1 $ such that
\begin{equation}\label{5.11}
\mathop{\inf}\limits_{j\geq1}\lambda_j=\lambda_{j_0}
=\left(\frac{\lvert\lvert u_n\rvert \rvert_{L^2}}{\lvert\lvert U^{j_0}_n\rvert \rvert_{L^2}}\right)^{\frac{2}{3}}.
\end{equation}
For $j=1,...,l$, put
\begin{equation}\label{5.12}
\widetilde{U}^j_n=U^j_n(\lambda^{-1}_j\ x), \quad \widetilde{u}^l_n=u^l_n((\lambda^{l}_n)^{-1}\ x).
\end{equation}
Then we have that
\begin{equation}\label{5.13}
\lvert\lvert \widetilde{U}^j_n\rvert \rvert^2_{L^2}=m=\lvert\lvert \widetilde{u}^l_n\rvert \rvert^2_{L^2},
\end{equation}
\begin{equation}\label{5.14}
E(U^j_n)=\frac{E(\widetilde{U}^j_n)}{\lambda^3_j}+\frac{1-\lambda^{-2}_j}{2}\int\lvert \nabla U^j_n\rvert^2dx,
\end{equation}
\begin{equation}\label{5.15}
E(u^l_n)=\frac{E(\widetilde{u}^l_n)}{(\lambda^l_n)^3}+\frac{1-(\lambda^l_n)^{-2}}{2}\int\lvert \nabla u^l_n\rvert^2dx.
\end{equation}
Now for $j=1,...,l$, put
\begin{equation}
V_{n}^{j}=\mu_{j}^{\frac{3}{2}}\widetilde{U}^j_n(\mu_{j} x),\;\;\; v_{n}^{l}=(\mu_{n}^{l})^{\frac{3}{2}}\widetilde{u}^l_n(\mu_{n}^{l} x).
\end{equation}
Then by Proposition \ref{p3.2}, there exist $\mu_{j}>0\;(j=1,...,l)$ and $\mu_{n}^{l}>0$ such that
\begin{equation}
I(V_{n}^{j})=0,\;\;\; I(v_{n}^{l})=0,
\end{equation}
\begin{equation}
\lvert\lvert V_{n}^{j}\lvert\lvert_{L^{2}}^{2} =m=\lvert\lvert v_{n}^{l}\lvert\lvert_{L^{2}}^{2},
\end{equation}
\begin{equation}
E(\widetilde{U}^j_n)\geq E(V_{n}^{j}),\;\;\; E(\widetilde{u}^l_n)\geq E(v_{n}^{l}).
\end{equation}
Thus we deduce that as $n\rightarrow\infty$ and $l\rightarrow\infty$,
\begin{equation}\label{5.16}
E(u_n)\geq d_m^{I}+\mathop{\inf}\limits_{j\geq1}(\frac{1-\lambda^{-2}_j}{2})\sum^l_{j=1}\int\lvert\nabla U^j_n\rvert^2dx+\frac{1-(\lambda^l_n)^{-2}}{2}\int\lvert\nabla u^l_n\rvert^2dx+o(1).
\end{equation}
Let $\beta=min\{\lambda_{j_0},\lambda^l_n\}$. Then by Theorem \ref{t5.1} we have that
\begin{equation}\label{5.17}
H(u_n)\geq d_m^{I}+\frac{1-\beta^{-2}}{2}c+o(1).
\end{equation}
It follows that
\begin{equation}\label{5.18}
d_m^{I} \geq d_m^{I}+\frac{1-\beta^{-2}}{2}c.
\end{equation}
Since $c$ is a positive constant, it follows that $\beta\leq1$. Thus we get that
\begin{equation}\label{5.19}
\lvert\lvert u_n\rvert \rvert_{L^2}\leq \lvert\lvert U^{j_0}_n\rvert \rvert_{L^2} \quad or \quad
\lvert\lvert u_n\rvert \rvert_{L^2}\leq \lvert\lvert u^{l}_n\rvert \rvert_{L^2}.
\end{equation}
If $\lvert\lvert u_n\rvert \rvert_{L^2}\leq\lvert\lvert u^{l}_n\rvert \rvert_{L^2}$, one deduces that
\begin{equation}\label{5.20}
\lvert\lvert u_n\rvert \rvert^4_{L^4(\mathbb{R}^3)}\rightarrow0 \quad as \quad n\rightarrow\infty.
\end{equation}
This is contradictory with Theorem \ref{t5.1}.
Therefore it is necessary that
\begin{equation}\label{5.20c}
\lvert\lvert u_n\rvert \rvert_{L^2}\leq\lvert\lvert U^{j_0}_n\rvert \rvert_{L^2}.
\end{equation}
Thus we get that
\begin{equation}\label{5.21}
\lvert\lvert u_n\rvert \rvert^2_{L^2}=\lvert\lvert U^{j_0}_n\rvert \rvert^2_{L^2}, \quad
\lvert\lvert \nabla u_n\rvert \rvert^2_{L^2}=\lvert\lvert \nabla U^{j_0}_n\rvert \rvert^2_{L^2}.
\end{equation}
By (\ref{5.3}), (\ref{5.5}) and (\ref{5.6}), it follows that
\begin{equation}\label{5.21aa}
u_n(x)=U^{j_0}_n(x)=U^{j_0}(x-x^{j_0}_n).
\end{equation}
Let
\begin{equation}\label{5.21a}
u_n\rightharpoonup v \ in  \ H^1(\mathbb{R}^3).
\end{equation}
Then
\begin{equation}\label{5.21b}
u_n\rightarrow v \ a.e.\ in \ \mathbb{R}^3.
\end{equation}
Thus there exists some fixed $x^{j_0}_n$, denoted by $x^{j_0}$ such that
\begin{equation}\label{5.21c}
v=U^{j_0}(x-x^{j_0}):=U^{j_0}, \ a.e. \ in\ \mathbb{R}^3.
\end{equation}
Therefore we have that
\begin{equation}\label{5.22}
u_n\rightarrow U^{j_0}\ \  in \ \ H^1(\mathbb{R}^3).
\end{equation}
It is clear that $I(U^{j_0})=0$. Thus $U^{j_0}$ is a minimizer of (\ref{1.31}). Then for some $y\in\mathbb{R}^3$ and $\theta\in\mathbb{R}$,
\begin{equation}\label{5.22a}
U^{j_0}=\psi(\cdot+y)e^{i\theta}.
\end{equation}
Therefore
\begin{equation}\label{5.23}
u_n\rightarrow \psi(\cdot+y)e^{i\theta}\ \  in \ \ H^1(\mathbb{R}^3).
\end{equation}

This completes the proof of Theorem \ref{t5.2}.
\end{proof}
\begin{definition}\label{d5.3}
Let $\omega\in(0,\frac{3}{16})$ and $Q_\omega$ be the ground state of (\ref{1.2}). We call the solution $Q_\omega(x)e^{i\omega t}$ of (\ref{1.1}) holds the orbital stability  if arbitrary $\varepsilon>0$, there exists $\delta>0$ such that for any $\varphi_0\in H^1(\mathbb{R}^3)$, when
\begin{equation*}
\mathop{\inf}\limits_{\{\theta\in\mathbb{R},y\in\mathbb{R}^3\}}\lvert\lvert \varphi_0(\cdot)-e^{i\theta}Q_\omega(\cdot+y)\rvert \rvert_{H^1}<\delta,
\end{equation*}
the solution $\varphi(t,x)$ of (\ref{1.1}) with $\varphi(0,x)=\varphi_0(x)$ satisfies
\begin{equation*}
\mathop{\inf}\limits_{\{\theta\in\mathbb{R},y\in\mathbb{R}^3\}}\lvert\lvert \varphi(t,\cdot)-e^{i\theta}Q_\omega(\cdot+y)\rvert \rvert_{H^1}<\varepsilon,\quad t\in\mathbb{R}.
\end{equation*}
\end{definition}

\begin{remark}\label{r5.4}
By Definition \ref{d5.3},  the ground state $Q_\omega(x)$ of (\ref{1.2}) holds the orbital stability if and only if the soliton $Q_\omega(x)e^{i\omega t}$ of (\ref{1.1})  holds the orbital stability. Moreover  all solitons in (\ref{1.10}) hold the orbital stability if $Q_\omega(x)$ holds the orbital stability.
\end{remark}
\begin{theorem}\label{t5.5}
Let $m\geq \frac{4}{3\sqrt{3}}\rho$, $\psi$ be the unique positive minimizer of (\ref{1.31}) up to translations and $v$ be the unique rescaling function by $\psi(x)=av(\lambda x)$ with (\ref{3.2}). Then the energy ground state $v(x)$ holds the orbital stability.
\end{theorem}
\begin{proof}
By Proposition \ref{p2.16} and Theorem \ref{t4.1}, for $m\geq \frac{4}{3\sqrt{3}}\rho$, the variational problem (\ref{1.31}) possesses a unique positive minimizer $\psi$ up to a translation. Then $\psi$ satisfies the Euler-Lagrange equation (\ref{3.1}) with the Lagrange multipliers $\mu\geq 0$ and $\nu>0$. Under the change of variables $\psi(x)=av(\lambda x)$ with (\ref{3.2}), $v=v(x)$ is the ground state of (\ref{1.2}). By Proposition \ref{p2.1} (refer to \cite{Z2006}), for arbitrary $\varphi_{0}\in H^{1}$, (\ref{1.1}) with $\varphi(0,x)=\varphi_{0}(x)$ possesses a unique global solution $\varphi(t,x)\in C(\mathbb{R}, H^{1}(\mathbb{R}^{3}))$. In addition, $\varphi(t,x)$ satisfies the mass conservation $M(\varphi(t,\cdot))=M(\varphi_{0}(\cdot))$ and the energy conservation $E(\varphi(t,\cdot))=E(\varphi_{0}(\cdot))$ for all $t\in\mathbb{R}$. Now arguing by contradiction.

If the conclusion of Theorem \ref{t5.5} does not hold, then there exist $\varepsilon>0$, a sequence $(\varphi_{0}^{n})_{n\in N^{+}}$ such that
\begin{equation}\label{5.25}
\mathop{\inf}\limits_{\{\theta\in\mathbb{R},y\in\mathbb{R}^3\}}\lvert\lvert \varphi_0^n-e^{i\theta}v(\cdot+y)\rvert \rvert_{H^1}<\frac{1}{n},
\end{equation}
and a sequence $(t_n)_{n\in \mathbb{N}^+}$ such that
\begin{equation}\label{5.26}
\mathop{\inf}\limits_{\{\theta\in\mathbb{R},y\in\mathbb{R}^3\}}\lvert\lvert \varphi_n(t_n,\cdot)-e^{i\theta}v(\cdot+y)\rvert \rvert_{H^1}\geq\varepsilon,
\end{equation}
where $\varphi_n$ denotes the global solution of (\ref{1.1}) with $\varphi(0,x)=\varphi_0^n$. From (\ref{5.25}) it yields that for some $\theta\in\mathbb{R}$, $y\in\mathbb{R}^3$,
\begin{equation}\label{5.27}
\varphi^n_0\rightarrow e^{i\theta}v(\cdot+y), \quad in \quad H^1(\mathbb{R}^3), \quad n\rightarrow\infty.
\end{equation}
Thus we have that
\begin{equation}\label{5.28}
\int\lvert \varphi^n_0\rvert^2dx \rightarrow \int  v^2dx,\;\; \quad E(\varphi^n_0)\rightarrow E(v),\;\; n\rightarrow\infty.
\end{equation}
Additionly we have that
\begin{equation}\label{5.29}
I(v)=0.
\end{equation}
Since for $n\in N^{+}$ ,
\begin{equation}\label{5.30}
\int\lvert \varphi_{n}(t_{n},\cdot)\lvert^{2}dx=\int \lvert \varphi_{0}^{n}\lvert^{2}dx,\;\;\; E(\varphi_{n}(t_{n},\cdot))=E(\varphi_{0}^{n}),
\end{equation}
from (\ref{5.28}) we have that
\begin{equation}\label{5.31}
\int\lvert \varphi_{n}(t_{n},\cdot)\lvert^{2}dx\rightarrow \int v^{2}dx,\;\;\; E(\varphi_{n}(t_{n},\cdot))\rightarrow E(v),\;\; n\rightarrow\infty.
\end{equation}
Now we conduct the change of varibles
\begin{equation}\label{5.32}
\phi_{n}(x)=a\varphi_{n}(t_{n},\lambda x)(n=1,2,\cdot\cdot\cdot),\;\;\; \psi(x)=av(\lambda x),
\end{equation}
with (\ref{3.2}). By (\ref{5.31}) we have that
\begin{equation}\label{5.33}
\int\lvert \phi_{n}\lvert^{2}dx\rightarrow \int \psi^{2}dx,\;\;\; E(\phi_{n})\rightarrow E(\psi),\;\; n\rightarrow\infty.
\end{equation}
We conduct the scaling
\begin{equation}\label{5.34}
\phi_{n}^{\lambda_{n}}=\lambda_{n}^{\frac{3}{2}}\phi_{n}(\lambda_{n}x)\quad with \quad \lambda_{n}>0.
\end{equation}
Then by Proposition \ref{3.2}, there exist $\lambda_{n}>0(n=1,2,\cdot\cdot\cdot)$ such that for $n=1,2,\cdot\cdot\cdot$,
\begin{equation}\label{5.35}
I(\phi_{n}^{\lambda_{n}})=0,\;\;\; \int \lvert \phi_{n}\lvert^{2}dx= \int \lvert \phi_{n}^{\lambda_{n}}\lvert^{2}dx,\;\;\; E(\phi_{n}^{\lambda_{n}})\leq E(\phi_{n}).
\end{equation}
Thus $(\phi_{n}^{\lambda_{n}})_{n\in N^{+}}$ is a minimizing sequence of (\ref{1.31}). Therefore
\begin{equation}\label{5.36}
\lim_{n\rightarrow\infty}\lvert\lvert \phi_{n}^{\lambda_{n}}-e^{i\theta}\psi(\cdot+y)\rvert \rvert_{H^1}=0.
\end{equation}
This yields that from (\ref{5.32})
\begin{equation}\label{5.37}
\lim_{n\rightarrow\infty}\lvert\lvert \varphi_{n}(t_{n},\cdot)-e^{i\theta}v(\cdot+y)\rvert \rvert_{H^1}=0.
\end{equation}
This is contradictory with (\ref{5.26}).

Theorem \ref{t5.5} is proved.

\end{proof}

\begin{theorem}\label{t5.6}
Arbitrary $\omega\in \mathcal{B}$, one has that $Q_{\omega}$ is orbitally stable. Especially $Q_{\omega^*}$ is orbitally stable.
\end{theorem}
\begin{proof}
According to Theorem \ref{t5.5}, if $\omega\in \mathcal{B}$, then $Q_{\omega}$ is the energy ground state. Therefore $Q_{\omega}$ is orbitally stable. By Theorem \ref{t4.16}, $\omega_{\ast}\in \mathcal{B}$, thus $Q_{\omega_{\ast}}$ is orbitally stable.

This proves Theorem \ref{t5.6}.
\end{proof}

\section{Spectral approach }\label{sec11}
Let $\omega\in(0,\frac{3}{16})$ and $Q_\omega(x)$ be the unique positive solution of (\ref{1.2}). Then one has that
\begin{equation}\label{6.1}
-\Delta Q_\omega+\omega Q_\omega-Q_\omega^3+Q_\omega^5=0, \quad Q_\omega\in H^1(\mathbb{R}^3).
\end{equation}
The linearized operator of (\ref{6.1}) around $Q_\omega$ is that
\begin{equation}\label{6.2}
L_\omega=-\Delta+\omega-3Q^2_\omega+5Q^4_\omega.
\end{equation}
It is clear that
\begin{equation}\label{6.3}
L_\omega=E^{\prime\prime}(Q_\omega)+\frac{1}{2}\omega M^{\prime\prime}(Q_\omega).
\end{equation}

\begin{proposition}\label{p6.1}
The operator $L_\omega$ has one negative simple eigenvalue and has its kernel spanned by $iQ_\omega$. Moreover the positive spectrum of $L_\omega$ is bounded away from zero.
\end{proposition}
\begin{proof}
Since $\omega\in(0,\frac{3}{16})$, by Proposition \ref{p2.2}, there exists a unique positive radial function
$Q_\omega(x)$ satisfying (\ref{6.1}). Now suppose that $\lambda\in\mathbb{R}$ satisfies $L_\omega Q_\omega=\lambda Q_\omega$, that is
\begin{equation}\label{6.4}
-\Delta Q_\omega+\omega Q_\omega-3Q_\omega^3+5Q_\omega^5=\lambda Q_\omega.
\end{equation}
By (\ref{6.1}) we have that
\begin{equation}\label{6.5}
-2Q_\omega^3+4Q_\omega^5=\lambda Q_\omega.
\end{equation}
Thus we can uniquely determine $\lambda$ as follows
\begin{equation}\label{6.6}
\lambda=\lambda_*=(\int -2Q^4_\omega+4Q^6_\omega dx)/ \int Q^2_\omega dx.
\end{equation}
On the other hand, from (\ref{6.5}) we have that
\begin{equation}\label{6.7}
\lambda=2Q_\omega^2(-1+2Q^2_\omega).
\end{equation}
According to Proposition \ref{p2.2},
\begin{equation}\label{6.8}
\lim_{\lvert x\rvert\rightarrow\infty}Q_\omega(x)=0.
\end{equation}
Then (\ref{6.7}) and (\ref{6.8}) yield that $\lambda<0$. Thus $\lambda_*<0$. Therefore we get that $L_\omega$ has one negative simple eigenvalue $\lambda_*$. In addition
\begin{equation}\label{6.8a}
L_\omega Q_\omega=\lambda_*Q_\omega.
\end{equation}
By the uniqueness of $Q_\omega$ we get that the kernel is spanned by $Q_\omega$. It follows that the kernel is spanned by $iQ_\omega$.

Now suppose that $\lambda>0$ and $u\in H^1(\mathbb{R}^3)\backslash \{0\}$ satisfying $L_\omega u=\lambda u$, that is
\begin{equation}\label{6.9}
-\Delta u+\omega u-3Q_\omega^2 u+5Q_\omega^4 u=\lambda u.
\end{equation}
By Proposition \ref{p2.2}, it follows that
\begin{equation}\label{6.10}
-3Q_\omega^2 +5Q_\omega^4 :=g(x)=o(\lvert x\rvert^{-1}).
\end{equation}
From Kato \cite{K1959}, $-\Delta+g(x)$ has no positive eigenvalues. Thus (\ref{6.9}) derives that $\lambda\leq\omega$. By Weyl's theorem on the essential spectrum, the rest of the spectrum of $L_\omega$ is bounded away from zero (see \cite{RS1978}).

This proves Proposition \ref{p6.1}.
\end{proof}

By Proposition \ref{p6.1}, $L_\omega$ with $T^{\prime}(0)=i$ satisfies Assumption 3 in Grillakis-Shatah-Strauss \cite{GSS1987} for $\omega\in(0,\frac{3}{16})$. With $J=-i$, $X=H^1(\mathbb{R}^3)$
and  $E$ defined as (\ref{1.7}), by Proposition \ref{p2.1} and Proposition \ref{p2.2}, (\ref{1.1}) satisfies
Assumption 1 and Assumption 2 in Grillakis-Shatah-Strauss \cite{GSS1987} for $\omega\in(0,\frac{3}{16})$.
Thus in terms of \cite{GSS1987}, we have the following propositions.
\begin{proposition}\label{p6.2}
Let $\omega\in(0,\frac{3}{16})$ and $Q_\omega$ be the ground state of (\ref{1.2}). If $Q_\omega$ satisfies
$$\frac{d}{d\omega}\int Q^2_\omega dx>0,$$
then $Q_\omega$ is a positive minimizer of the variational problem
\begin{equation*}
d_\omega=\mathop{\inf}\limits_ {\{u\in H^1(\mathbb{R}^3),\  \int\lvert u\rvert^2dx=\int Q_\omega^2dx\}
}
E(u).
\end{equation*}
\end{proposition}
\begin{proposition}\label{p6.3}
Let $\omega\in(0,\frac{3}{16})$ and $Q_\omega$ be the ground state of (\ref{1.2}). If $Q_\omega$ satisfies
$$\frac{d}{d\omega}M(Q_\omega)=\frac{d}{d\omega}\int Q^2_\omega dx>0,$$
then $Q_\omega$ is orbitally stable. If $Q_\omega$ satisfies that
$$\frac{d}{d\omega}M(Q_\omega)=\frac{d}{d\omega}\int Q^2_\omega dx<0,$$
then $Q_\omega$ is orbitally unstable.
\end{proposition}
\begin{theorem}\label{t6.4}
Let $\omega\in(0,\frac{3}{16})$ and $Q_\omega$ be the ground state of (\ref{1.2}). If $\omega_0\in(0,\frac{3}{16})$ satisfies that $$\frac{d}{d\omega}M(Q_\omega)\mid_{\omega=\omega_0}=0$$
and
$$\frac{d}{d\omega}M(Q_\omega)<0 \ for \ \omega\in(\omega_0-\delta_0,\omega_0+\delta_0)\backslash \{\omega_0\}$$
with some $\delta_0>0$, then the soliton $Q_{\omega_0}e^{i\omega_0t}$ of (\ref{1.1}) is orbitally unstable.
\end{theorem}
\begin{proof}
 For $\omega\in(0,\frac{3}{16})$, put
\begin{equation*}
d(\omega)=E(Q_\omega)+\omega M(Q_\omega).
\end{equation*}
Then
\begin{equation*}
d^{\prime\prime}(\omega)=\frac{d}{d\omega}M(Q_\omega).
\end{equation*}
By the assumptions of Theorem \ref{t6.4}, it follows that $d(\omega)$  is not convex in $(\omega_0-\delta_0,\omega_0+\delta_0)$. From Corollary 5.1 in \cite{GSS1987}, we get that the soliton
$Q_{\omega_0}e^{i\omega_0t}$ of (\ref{1.1}) is orbitally unstable.

This proves Theorem \ref{t6.4}.
\end{proof}

\section{Instability of  solitons}\label{sec12}

\begin{proposition}\label{p7.1}
Let $f(z)$ be real analytic on $(a,b)$ and $f(z)\not\equiv 0$ for $z\in(a,b)$. If $z_0\in(a,b)$ satisfies $f(z_0)=0$, then $z_0$ is an isolated zero point of $f(z)$, which means that there exists $\delta>0$ such that
$$f(z)\neq 0 \  for \ z\in(z_0-\delta,z_0+\delta)\backslash\{z_0\}.$$
\end{proposition}
\begin{proof}
Since $f(z_0)=0$, we assume that
\begin{equation}\label{7.1}
f(z_0)=f^{\prime}(z_0)=\cdots=f^{(n-1)}(z_0)=0 \ and \ f^{(n)}(z_0)\neq 0,\ n\geq 1.
\end{equation}
It follows that
\begin{equation}\label{7.2}
f(z)=\frac{f^{(n)}(z_0)}{n!}(z-z_0)^n+\frac{f^{(n+1)}(z_0)}{(n+1)!}(z-z_0)^{n+1}+\cdots=(z-z_0)^n g(z).
\end{equation}
It is clear that $g(z_0)\neq0$, $g(z)$ is continuous and  analytic on $(a,b)$. Thus there exists $\delta>0$ such that
$$g(z)\neq0 \ for\  z\in(z_{0}-\delta,z_0+\delta).$$
From (\ref{7.2}) we have that
\begin{equation}\label{7.3}
f(z_0)=0,\ f(z)\neq0 \ for \ z\in(z_0-\delta,z_0+\delta)\backslash\{z_0\}.
\end{equation}
Therefore $z_0$ is an isolated zero point of $f(z)$.

This proves Proposition \ref{p7.1}.
\end{proof}
\begin{proposition}\label{p7.2}
Let $\omega\in(0,\frac{3}{16})$ and $Q_\omega$ be the ground state of (\ref{1.2}). If $\omega_0\in(0,\frac{3}{16})$ satisfies that $$\frac{d}{d\omega}M(Q_\omega)\mid_{\omega=\omega_0}=0,$$
then $\omega_0$ is an isolated zero point of $\frac{d}{d\omega}M(Q_\omega)$.
\end{proposition}
\begin{proof}
In terms of Proposition \ref{p2.2a}, $M(Q_\omega)=\int Q^2_\omega dx$ is real analytic on $(0,\frac{3}{16})$. It follows that $\frac{d}{d\omega}M(Q_\omega)$ is also analytic on $(0,\frac{3}{16})$. Now suppose that
$$\frac{d}{d\omega}M(Q_\omega)\equiv0 \ for \ \omega\in(0,\frac{3}{16}).$$
Then it leads to
$$ M(Q_\omega)\equiv constant=c \ for \ \omega\in(0,\frac{3}{16}).$$
This is contradictory  with Proposition \ref{p2.2a}. Therefore $$\frac{d}{d\omega}M(Q_\omega)\not\equiv 0 \ for \ \omega\in(0,\frac{3}{16}).$$
By Proposition \ref{p7.1}, one deduces Proposition \ref{p7.2}.

This proves Proposition \ref{p7.2}.
\end{proof}
\begin{theorem}\label{t7.3}
Let $\omega\in(0,\frac{3}{16})$ and $Q_\omega$ be the ground state of (\ref{1.2}). Then we have that for $\omega\in (0,\omega_*)$,
$$\frac{d}{d\omega}M(Q_\omega)\leq 0. $$
Moreover if $\omega_0\in(0,\omega_*)$ satisfies  $$\frac{d}{d\omega}M(Q_\omega)\mid_{\omega=\omega_0}=0,$$
then there exists $\delta>0$ such that for $\omega\in(\omega_0-\delta,\omega_0+\delta)/ \{\omega_0\}$,
$$\frac{d}{d\omega}M(Q_\omega)<0.$$
\end{theorem}
\begin{proof}
By Lemma \ref{p2.17},
$$M(Q_\omega)\geq m_{0} \;\; for \;\;all\;\; \omega\in(0,\frac{3}{16}).$$
From Theorem \ref{t4.25},  $\omega_*\in(0,\frac{3}{16})$ satisfies $M(Q_{\omega_*})=m_{0}$. It follows that
$$\frac{d}{d\omega}M(Q_\omega)\mid_{\omega=\omega_*}=0.$$
From Proposition \ref{p7.2}, $\omega_*$ is an isolated zero point of $\frac{d}{d\omega}M(Q_\omega)$. From Proposition \ref{p2.2a} thus there exists the minimum $\omega_1\in[0,\omega_*)$ satisfying the following alternatives. If $\omega_1=0$, then for all $\omega\in(0,\omega_*)$ one has that
$$\frac{d}{d\omega}M(Q_\omega)<0.$$
If $\omega_1>0$, one has that for $\omega\in(\omega_1,\omega_*)$
$$\frac{d}{d\omega}M(Q_\omega)<0$$
and
$$\frac{d}{d\omega}M(Q_\omega)\mid_{\omega=\omega_1}=0.$$
In the following  we prove that with some $\delta>0$,
$$\frac{d}{d\omega}M(Q_\omega)<0 \ for \ \omega\in(\omega_1-\delta,\omega_1).$$

Otherwise, from Proposition \ref{p7.2} there exists $\delta_1>0$ such that $$\frac{d}{d\omega}M(Q_\omega)>0 \ for \
\omega\in(\omega_1-\delta_1,\omega_1).$$
By Proposition \ref{p2.2a},
$$M(Q_\omega)\in C^1(0,\frac{3}{16})\ and \ \lim\limits_{\omega\rightarrow\frac{3}{16}}M(Q_\omega)=\infty.$$
Thus there must exist $\omega_2\in(\omega_*,\frac{3}{16})$ such that $$M(Q_{\omega_2})=M(Q_{\omega_1})$$
and with some $\delta>0$
\begin{equation}\label{7.3a}
\frac{d}{d\omega}M(Q_\omega)\geq 0 \ for \ \omega\in(\omega_2-\delta,\omega_2].
\end{equation}
Now we divide two situations to proceed.

The first situation: $\frac{d}{d\omega}M(Q_\omega)\mid_{\omega=\omega_2}>0$. Then there exists enough small $\delta_2>0$
such that
$$\frac{d}{d\omega}M(Q_\omega)\mid_{\omega=\omega_2-\delta_2}>0.$$

The second situation: $\frac{d}{d\omega}M(Q_\omega)\mid_{\omega=\omega_2}=0$. By Proposition \ref{p7.2} and (\ref{7.3a}),
there exists enough small $\delta_2>0$ such that $$\frac{d}{d\omega}M(Q_\omega)\mid_{\omega=\omega_2-\delta_2}>0.$$
It returns to the first situation.

In both situations, there must exist $\omega^*\in(\omega_1-\delta_1,\omega_1)$ such that
$$M(Q_{\omega^*})=M(Q_{\omega_2-\delta_2}) \ \  and \ \
\frac{d}{d\omega}M(Q_\omega)\mid_{\omega=\omega^*}>0.$$
In terms of Proposition \ref{p6.2}, both $Q_{\omega^*}$ and $Q_{\omega_2-\delta_2}$ are positive minimizers of the variational problem
\begin{equation*}
d_{\omega^*}:=\mathop{\inf}\limits_ {\{u\in H^1(\mathbb{R}^3),\  \int\lvert u\rvert^2dx=M(Q_{\omega^*})=M(Q_{\omega_2-\delta_2})\}
}
E(u).
\end{equation*}
From Theorem \ref{t4.13}, $Q_{\omega^*}\equiv Q_{\omega_2-\delta_2}$ up to a translation. This is contradictory with $\omega^*\neq \omega_2-\delta_2$. Thus we conclude that if $\omega_1>0$, then $$\frac{d}{d\omega}M(Q_\omega)\mid_{\omega=\omega_1}=0$$
and with some $\delta>0$
$$\frac{d}{d\omega}M(Q_\omega)<0 \ for \ \omega\in(\omega_1-\delta,\omega_1)\cup(\omega_1,\omega_*).$$
Since the zero points of $\frac{d}{d\omega}M(Q_\omega)$ on $(0,\omega_*)$ are isolated and countable, this process can go on and on.  By mathematical induction we deduce the following fact. For arbitrary $\omega\in(0,\omega_*)$, one has that
$$\frac{d}{d\omega}M(Q_\omega)\leq 0.$$
If $\omega_0\in(0,\omega_*)$ satisfies that $$\frac{d}{d\omega}M(Q_\omega)\mid_{\omega=\omega_0}=0,$$
 there exists $\delta>0$ such that
$$\frac{d}{d\omega}M(Q_\omega)<0 \ for \ \omega\in(\omega_0-\delta,\omega_0+\delta)/ \{\omega_0\}.$$

This proves Theorem \ref{t7.3}.
\end{proof}
\begin{theorem}\label{t7.4}
Let $\omega\in(0,\frac{3}{16})$ and $Q_\omega$ be the ground state of (\ref{1.2}).
Then for $\omega\in(0,\omega_*)$, the soliton $Q_\omega e^{i\omega t}$ of (\ref{1.1}) is orbitally unstable.
\end{theorem}

\begin{proof}
When $0<\omega<\omega_*$, by Theorem \ref{t7.3},
$$\frac{d}{d\omega}M(Q_\omega)<0,$$
or
$$\frac{d}{d\omega}M(Q_\omega)=0$$
but $\omega$ is an isolated zero point. In the following we divide two cases to proceed.

The first case: $0<\omega<\omega_*$ and $\frac{d}{d\omega}M(Q_\omega)<0$. In this case, from  Proposition \ref{p6.3}, the soliton $Q_\omega e^{i\omega t}$ of (\ref{1.1}) is orbitally unstable.

The second case: $0<\omega<\omega_*$ and $\frac{d}{d\omega}M(Q_\omega)=0$. In this case, we let $\omega_0\in(0,\omega_*)$ satisfies
$$\frac{d}{d\omega}M(Q_\omega)\mid_{\omega=\omega_0}=0.$$
By Theorem \ref{t7.3},
there exists $\delta_0>0$ such that
$$\frac{d}{d\omega}M(Q_\omega)<0 \ for \ \omega\in(\omega_0-\delta,\omega_0+\delta)/ \{\omega_0\}.$$
From Theorem \ref{t6.4}, the soliton $Q_{\omega_0} e^{i\omega_0 t}$ of (\ref{1.1}) is orbitally unstable.

Therefore it is shown that for any $\omega\in(0,\omega_*)$, the soliton $Q_{\omega} e^{i\omega t}$ of (\ref{1.1}) is orbitally unstable.

This completes the proof of Theorem \ref{t7.4}.
\end{proof}
\begin{theorem}\label{t7.5}
Let $\omega\in(0,\frac{3}{16})$ and $Q_\omega$ be the ground state of (\ref{1.2}).
Then the map $\omega\rightarrow M(Q_\omega)$ is strictly decreasing for $\omega<\omega_*$.
\end{theorem}

\begin{proof}
 By Theorem \ref{t7.3},
$$\frac{d}{d\omega}M(Q_\omega)\leq0 \  for \ \omega<\omega_*$$
and the zero points of $\frac{d}{d\omega}M(Q_\omega)$ are isolated and countable. Since $M(Q_\omega)\in C^1(0,\frac{3}{16})$, it follows that the map $\omega\rightarrow M(Q_\omega)$ is strictly decreasing for $\omega<\omega_*$ .

This proves Theorem \ref{t7.5}.

\end{proof}

\section{Stability of  solitons}\label{sec13}

\begin{theorem}\label{t8.1}
Let $\omega\in(0,\frac{3}{16})$ and $Q_\omega$ be the ground state of (\ref{1.2}).
Then we have that for $\omega>\omega_*$,
$$\frac{d}{d\omega}M(Q_\omega)\geq0 .$$
Moreover
if $\omega_0\in(\omega_*,\frac{3}{16})$ satisfies $$\frac{d}{d\omega}M(Q_\omega)\mid_{\omega=\omega_0}=0,$$
then there exists $\delta>0$ such that for $ \omega\in(\omega_0-\delta,\omega_0+\delta)/ \{\omega_0\}$,
$$\frac{d}{d\omega}M(Q_\omega)>0 .$$
\end{theorem}
\begin{proof}
By Lemma \ref{p2.17},
$$M(Q_\omega)\geq m_{0} \;\; for  \;\; all\;\; \omega\in(0,\frac{3}{16}).$$
Since $\omega_*\in(0,\frac{3}{16})$ satisfies $M(Q_{\omega_*})=m_{0}$. It follows that
$$\frac{d}{d\omega}M(Q_\omega)\mid_{\omega=\omega_*}=0.$$
From Proposition \ref{p7.2}, $\omega_*$
is an isolated zero point of $\frac{d}{d\omega}M(Q_\omega)$. Thus from Proposition \ref{p2.2a}, there exists the maximum $\omega_1\in(\omega_*,\frac{3}{16}]$ satisfying the following alternatives. If $\omega_1=\frac{3}{16}$,
then one has that
$$\frac{d}{d\omega}M(Q_\omega)>0 \ for \  \omega\in(\omega_*,\frac{3}{16}).$$
If $\omega_1<\frac{3}{16}$, one has that
$$\frac{d}{d\omega}M(Q_\omega)>0 \ for \ \omega\in(\omega_*, \omega_1)$$
and
$$\frac{d}{d\omega}M(Q_\omega)\mid_{\omega=\omega_1}=0.$$
In the following we prove that with some $\delta>0$,
$$\frac{d}{d\omega}M(Q_\omega)>0 \ for \ \omega\in(\omega_1-\delta,\omega_1).$$

Otherwise, from Proposition \ref{p7.2}, there exists $\delta_1>0$ such that $$\frac{d}{d\omega}M(Q_\omega)<0 \ for \ \omega\in(\omega_1,\omega_1+\delta_1).$$
By Proposition \ref{p2.2a},
$$M(Q_\omega)\in C^1(0,\frac{3}{16}) \ and \ \lim\limits_{\omega\rightarrow\frac{3}{16}}M(Q_\omega)=\infty.$$
Thus there must exist $\omega_2\in(\omega_1+\delta_1,\frac{3}{16})$ such that $$M(Q_{\omega_2})=M(Q_{\omega_1})$$
and with some $\delta>0$
\begin{equation}\label{8.1a}
\frac{d}{d\omega}M(Q_\omega)\geq0 \ for \ \omega\in(\omega_2-\delta,\omega_2].
\end{equation}
Now we divide two situations to proceed.

The first situation: $\frac{d}{d\omega}M(Q_\omega)\mid_{\omega=\omega_2}>0$. Then there exists enough small $\delta_2>0$ such that $$\frac{d}{d\omega}M(Q_\omega)\mid_{\omega=\omega_2-\delta_2}>0.$$

The second situation: $\frac{d}{d\omega}M(Q_\omega)\mid_{\omega=\omega_2}=0$. By Proposition \ref{p7.2} and (\ref{8.1a}),
there exists enough small $\delta_2>0$ such that $$\frac{d}{d\omega}M(Q_\omega)\mid_{\omega=\omega_2-\delta_2}>0.$$
It returns to the first situation.

In both situations, there must exist  $\omega^*\in(\omega_*,\omega_1)$ such that
$$M(Q_{\omega^*})=M(Q_{\omega_2-\delta_2}) \ \  and \ \
\frac{d}{d\omega}M(Q_\omega)\mid_{\omega=\omega^*}>0.$$
In terms of Proposition \ref{p6.2}, both $Q_{\omega^*}$ and $Q_{\omega_2-\delta_2}$ are positive minimizers of the variational problem
\begin{equation*}
d_{\omega^*}:=\mathop{\inf}\limits_ {\{u\in H^1(\mathbb{R}^3),\  \int\lvert u\rvert^2dx=M(Q_{\omega^*})=M(Q_{\omega_2-\delta_2})\}
}
E(u).
\end{equation*}
From Theorem \ref{t4.13}, $Q_{\omega^*}\equiv Q_{\omega_2-\delta_2}$ up to a translation. This is contradictory with $\omega^*\neq\omega_2-\delta_2$. Thus we conclude that if $\omega_1<\frac{3}{16}$, then
$$\frac{d}{d\omega}M(Q_\omega)\mid_{\omega=\omega_1}=0$$
and with some $\delta>0$
$$\frac{d}{d\omega}M(Q_\omega)>0 \ for \ \omega\in(\omega_*,\omega_1) \cup (\omega_1,\omega_1+\delta).$$

This process can go on and on. Since the zero points of $\frac{d}{d\omega}M(Q_\omega)$ on $(\omega_*,\frac{3}{16})$ are isolated and countable, by mathematical induction we deduce the following fact. For arbitrary $\omega\in(\omega_*,\frac{3}{16})$, one has that $$\frac{d}{d\omega}M(Q_\omega)\geq 0.$$
If $\omega_0\in(\omega_*,\frac{3}{16})$ satisfies that
$$\frac{d}{d\omega}M(Q_\omega)\mid_{\omega=\omega_0}=0,$$
there exists $\delta>0$ such that
$$\frac{d}{d\omega}M(Q_\omega)> 0 \ for \ \omega\in(\omega_0-\delta,\omega_0+\delta)/ \{\omega_0\}.$$

This proves Theorem \ref{t8.1}.
\end{proof}
\begin{theorem}\label{t8.2}
Let $\omega\in(0,\frac{3}{16})$ and $Q_\omega$ be the ground state of (\ref{1.2}).
Then the map $\omega\rightarrow M(Q_\omega)$ is strictly increasing for $\omega>\omega_*$.
\end{theorem}

\begin{proof}
By Theorem \ref{t8.1},
$$\frac{d}{d\omega}M(Q_\omega)\geq0 \ for \ \omega>\omega_*,$$
and the zero points of  $\frac{d}{d\omega}M(Q_\omega)$ are isolated and countable. Since $M(Q_\omega)\in C^1(0,\frac{3}{16})$, we get that the map $\omega\rightarrow M(Q_\omega)$ is strictly increasing for $\omega>\omega_*$.

This proves Theorem \ref{t8.2}.
\end{proof}
\begin{theorem}\label{t8.3}
Let $\omega\in(0,\frac{3}{16})$ and $Q_\omega$ be the ground state of (\ref{1.2}).
Then for $\omega\in[\omega_*,\frac{3}{16})$,  the soliton $Q_\omega e^{i\omega t}$ of (\ref{1.1}) is orbitally stable.
\end{theorem}
\begin{proof}
When $\omega=\omega_*$, from Theorem \ref{t5.6}, the soliton $Q_{\omega_*} e^{i\omega_* t}$ of (\ref{1.1}) is orbitally stable.

When $\omega>\omega_*$, by Theorem \ref{t8.2},
$$\frac{d}{d\omega}M(Q_\omega)\geq0.$$
If $\omega$ satisfies
$$
\frac{d}{d\omega}M(Q_{\omega})=0,
$$
form Proposition \ref{p7.2}, one gets that $\omega$ is an isolated zero point. In the following we divide two cases to proceed.

The first case: $\omega>\omega_*$  and $\frac{d}{d\omega}M(Q_\omega)>0$. In this case, from  Proposition \ref{p6.3}, the soliton $Q_\omega e^{i\omega t}$ of (\ref{1.1}) is orbitally stable.

The second case: $\omega>\omega_*$  and $\frac{d}{d\omega}M(Q_\omega)=0$. In this case, we let $\omega_0\in(\omega_*,\frac{3}{16})$ satisfy $$\frac{d}{d\omega}M(Q_\omega)\mid_{\omega=\omega_0}=0.$$
For $Q_{\omega_{0}}$, from Proposition \ref{p2.18} one gets $R_{\omega_{0}}$. Let $M(R_{\omega_{0}})=m$. By Theorem \ref{t4.25}, $m\geq\frac{4}{3\sqrt{3}}\rho$. Now we construct the variational problem for $m=M(R_{\omega_{0}})$
\begin{equation}\label{8.300}
d_{m}^{I}=\inf_{\{u\in H^{1}(\mathbb{R}^{3}), \; M(u)=m, \; I(u)=0\}}E(u).
\end{equation}
By Proposition \ref{p2.16} and Theorem \ref{t4.1}, the variational problem (\ref{8.300}) possesses a unique energy minimizer $\psi$. Moreover $\psi$ satisfies the Euler-Lagrange equation (\ref{3.1}). We conduct the change of variables  $\psi=av(\lambda x)$ with (\ref{3.2}). Then $v$ is the energy ground state with $\omega$ in (\ref{3.2}). Let $\mathcal{B}$ is the set introduced in Section 4. It is clear that here $\omega\in\mathcal{B}$. By Theorem \ref{t5.5}, the energy ground state $v$ is orbitally stable. It follows that $\omega\notin (0, \omega_{\ast})$ from Theorem \ref{t7.4}. Thus $\omega\in [\omega_{\ast}, \frac{3}{16})$. According to Theorem \ref{t4.25}, one gets that $M(v)=M(Q_{\omega_{0})}$. From Theorem \ref{t8.2}, it yields that $v\equiv Q_{\omega_{0}}$ up to a translation and $\omega=\omega_{0}$. Since $v$ holds orbital stability, we get that $Q_{\omega_{0}}$ holds orbital stability.


Thus it is shown that for any $\omega\in[\omega_*,\frac{3}{16})$, the soliton $Q_\omega e^{i\omega t}$ of (\ref{1.1}) is orbitally stable.

This completes the proof of  Theorem \ref{t8.3}.
\end{proof}

According to Theorem \ref{t8.3} and Theorem \ref{t7.4}, sharp stability threshold of solitons for (\ref{1.1}) is established. Now we complete the proofs of Theorem \ref{t1.3} and Theorem \ref{t1.2}.

Proof of Theorem \ref{t1.3}: From Theorem \ref{t4.25}, $\omega_*$ is uniquely determined by $m_{0}$ and one has that $$\int \lvert Q_{\omega_{\ast}}\lvert^{2}dx=m_{0}.$$ Thus Theorem \ref{t7.4} and Theorem \ref{t8.3} deduce Theorem \ref{t1.2}.

Proof of Theorem \ref{t1.2}: For $\omega_{\ast}$ in Theorem \ref{t4.25}, Theorem \ref{t7.5} and Theorem \ref{t8.2} deduce Theorem \ref{t1.2}.

\section{Classification of normalized solutions}
\begin{theorem}\label{t9.1}
Let $\mathcal{B}$ and $\omega_*$ be the notations introduced in Section 4. Then
$$\mathcal{B}=[\omega_*,\frac{3}{16}).$$
\end{theorem}

\begin{proof}
By Theorem \ref{t4.16}, $\mathcal{B}\subset(0,\frac{3}{16})$. Thus Theorem \ref{t5.5} and
Theorem \ref{t7.5} derive that $\mathcal{B}\subset[\omega_*,\frac{3}{16})$. On the other hand, for arbitrary $\omega_{0}\in[\omega_*,\frac{3}{16})$, put $$m=\int Q_{\omega_{0}}^{2} dx.$$
For $Q_{\omega_{0}}$, from Proposition \ref{p2.18} one gets $R_{\omega_{0}}$. Let $M(R_{\omega_{0}})=m$. By Theorem \ref{t4.25}, one has $m\geq\frac{4}{3\sqrt{3}}\rho$. Now we construct the variational problem for $M=M(R_{\omega_{0}})$
\begin{equation*}
d_{m}^{I}=\inf_{\{u\in H^{1}(\mathbb{R}^{3}), \; M(u)=m, \; I(u)=0\}}E(u).
\end{equation*}
By Proposition \ref{p2.16} and Theorem \ref{t4.1}, the variational problem $d_{m}^{I}$ possesses a unique minimizer $\psi$. Moreover $\psi$ satisfies the Euler-Lagrange equation (\ref{3.1}). We conduct the change of variables $\psi=av(\lambda x)$ with (\ref{3.2}). Then $v$ is the energy ground state with $\omega$ in (\ref{3.2}). Let $\mathcal{B}$ is the set introduced in Section 4. It is clear that $\omega\in \mathcal{B}$. By Theorem \ref{t5.5}, the energy ground state $v$ is orbitally stable. It follows that $\omega\notin (0, \omega_{\ast})$ from Theorem \ref{t7.4}. Thus $\omega\in [\omega_{\ast}, \frac{3}{16})$. According to Theorem \ref{t4.25}, one gets that $M(v)=M(Q_{\omega_{0}})$. From Theorem \ref{t8.2}, it yields that $v\equiv Q_{\omega_{0}}$ up to a translation and $\omega=\omega_{0}$. Since $\omega\in \mathcal{B}$, we get that $\omega_{0}\in\mathcal{B}$. Thus $[\omega_*,\frac{3}{16})\subset \mathcal{B}$.

This proves Theorem \ref{t9.1}.

\end{proof}

\begin{theorem}\label{t9.2}
Let $\omega\in(0,\frac{3}{16})$ and $Q_\omega$ be the ground state of (\ref{1.2}). Then $M(Q_\omega)=\int Q^2_\omega dx$ establishes two 1-1   correspondences from $(0,\omega_*)$ to $(\infty, m_{0})$ and from $[\omega_*,\frac{3}{16})$ to $[m_{0},\infty)$ respectively.
\end{theorem}

\begin{proof}
By Proposition \ref{p2.2a} and Theorem \ref{t7.5}, $M(Q_\omega)$ establishes a 1-1 correspondences from $(0,\omega_*)$  to $(\infty, m_{0})$. By Proposition \ref{p2.2a} and Theorem \ref{t8.2}, $M(Q_\omega)$ establishes a 1-1 correspondences from $[\omega_*,\frac{3}{16})$ to $[m_{0},\infty)$.
By Theorem \ref{t4.16},
$$\int Q^2_{\omega_*}dx=m_{0}.$$

This proves Theorem \ref{t9.2}.
\end{proof}

\begin{theorem}\label{t9.3}
Let $\phi$ be the unique positive minimizer of (\ref{1.5}) and $\rho=\int \phi^2dx$.  Then there exists $m_{0}=\inf_{\{0<\omega<\frac{3}{16}\}}M(Q_{\omega})$ such that  $\frac{4}{3\sqrt{3}}\rho\leq m_{0}\leq\rho$. When $0\leq m<m_{0}$, (\ref{1.2}) has no any normalized solutions with the prescribed mass
$\int \lvert u\rvert^2dx=m$.
When $m=m_{0}$, (\ref{1.2}) has a unique positive normalized solution with the prescribed mass $\int \lvert u\rvert^2dx=m=m_{0}$. When $m>m_{0}$, (\ref{1.2}) has just two positive normalized solutions with the prescribed mass $\int \lvert u\rvert^2dx=m$.
\end{theorem}

\begin{proof}
By Proposition \ref{p2.2} and Lemma \ref{p2.17}, we get that when $0\leq m<\frac{4}{3\sqrt{3}}\rho$, (\ref{1.2}) has no any normalized solutions with the prescribed mass
$$\int \lvert u\rvert^2dx=m.$$
The rest of Theorem \ref{t9.3} can be directly gotten by Theorem \ref{t9.2}.

This proves Theorem \ref{t9.3}.
\end{proof}

Now we complete the proof of Theorem \ref{t1.4}.

Proof of Theorem \ref{t1.4}: By Theorem \ref{t9.2} and Theorem \ref{t9.3}, one deduces Theorem \ref{t1.4}.

\vspace{0.2cm}
\textbf{Acknowledgment.}

This research is supported by the National Natural Science Foundation of China 12271080.

\textbf{Data Availability Statement}

Data sharing not applicable to this article as no datasets were generated or analysed during the current study.




\end{document}